\documentclass{amsart} 
\usepackage[a4paper, total={6.5in, 9in}]{geometry}

\usepackage[english]{babel}
\usepackage{amsmath,amssymb,amsthm}
\usepackage[dvips]{graphicx}

\usepackage{natbib}
\usepackage[colorlinks,citecolor=blue,urlcolor=blue]{hyperref}
\usepackage{color}
\newcommand{\rosso}[1]{{#1}}

\makeatletter
\providecommand{\doi}[1]{%
  \begingroup
    \let\bibinfo\@secondoftwo
    \urlstyle{rm}%
    \href{http://dx.doi.org/#1}{%
      doi:\discretionary{}{}{}%
      \nolinkurl{#1}%
    }%
  \endgroup
}
\makeatother

\theoremstyle{plain}
\newtheorem{theorem}{Theorem}
\newtheorem{lemma}{Lemma}

\theoremstyle{definition}
\newtheorem{definition}{Definition}

\newtheorem{remark}{Remark}

\newcommand{\bbeta}{\boldsymbol{\beta}}
\newcommand{\Hbbeta}{\widehat{\bbeta}}
\newcommand{\HbbetaR}{\Hbbeta^{(\mathbf{R})}}

\allowdisplaybreaks

\title{Best estimation of functional linear models}

\author{Giacomo Aletti}
\address{ADAMSS Center \& Department of Mathematics, Universit\`a degli Studi di Milano, 20131 Milano, Italy}
\email{giacomo.aletti@unimi.it}

\author{Caterina May}
\address{University of Eastern Piedmont, Italy}
\email{caterina.may@unipmn.it}

\author{Chiara Tommasi}
\address{Universit\`a degli Studi di Milano, Italy}
\email{chiara.tommasi@unimi.it}

\begin{document}
\maketitle

\begin{abstract}
Observations which are realizations from some continuous process are frequent in sciences, engineering, economics, and other fields. We consider linear models, with possible random effects, where the responses are random functions in a suitable Sobolev space. The processes \rosso{cannot} be observed directly. With smoothing procedures from the original data, both the response curves and their derivatives can be reconstructed, even separately. From both these samples of functions, just one sample of representatives is obtained to estimate the vector of functional parameters. \rosso{A simulation study shows the benefits of this approach over the common method of using information either on curves or derivatives. The main theoretical result is a strong functional version of the Gauss-Markov theorem. This ensures that the proposed functional estimator is more efficient than the best linear unbiased estimator based only on curves or derivatives.}
\end{abstract}

\vspace*{0.5cm}


\section{Introduction}\label{sect:intro}
Observations which are realizations from some continuous process are ubiquitous in many fields like sciences, engineering, economics and other fields. For this reason, the interest for statistical modeling of functional data is increasing, with applications in many areas. Reference monographs on functional data analysis are, for instance, the books of \citet{Ramsay:Silverman05} and \citet{HorKok:HK12}, and the book of \citet{Ferr:V06} for the non-parametric approach. They cover topics like data representation, smoothing and registration; regression models; classification, discrimination and principal component analysis; derivatives and principal differential analysis; and \rosso{many} other.

Regression models with functional variables can cover different situations: it may be the case of functional responses, or functional predictors, or both. In the present paper linear models with functional response and multivariate (or univariate) regressor are considered. We consider the case of repeated measurements but all the theoretical results remain valid in the standard case. Focus of the work is the best estimation of the functional coefficients of the regressors.

The use of derivatives is very important for exploratory analysis of functional data as well as for inference and prediction methodologies.
High quality derivative information can be provided, for instance, by reconstructing the functions with spline smoothing procedures. Recent developments on estimation of derivatives are contained in  \citet{sangalli09} and in \citet{pigoli12}. See also \citet{Baraldo13}, who have obtained derivatives in the context of survival analysis, and \citet{Hall09} who have estimated derivatives in a non-parametric model.

Curves and derivatives are actually reconstructed from a set of observed values, because
the response processes cannot be observed directly.
In the literature the usual space  for functional data is $L^2$, and the observed values are used to reconstruct either curve functions or derivatives. 

To our knowledge, the most common method to reconstruct derivatives is to build the sample of functions by a smoothing procedure of the data, and then to differentiate these curve functions.  
However, the sample of functions and the sample of derivatives may be obtained separately.
For instance, different smoothing techniques may be used to obtain the  functions and the derivatives. Another possibility is when two sets of data are available, which are suitable to estimate functions and derivatives, respectively.

\rosso{Some possible examples of data concerning curves and derivatives are: in studying how the velocity of a car on a particular street is influenced by some covariates, the velocity is measured by a police radar; in addition we could benefit of more information since its position is tracked by a GPS. In chemical experiments, data on reaction velocity and concentration may be collected separately.}

The novelty of the present work is that both information on curves and derivatives (that are not obtained by  differentiation of the curves themselves) are used to estimate the functional coefficients.

The heuristic justification for this choice is that the data may provide different information on curve functions and their derivatives and it is always recommended to use the whole available information.
Actually, we prove that if we take into consideration both information about curves and derivatives, we obtain the best linear unbiased estimates for the functional coefficients. Therefore, the common method of using information on either  curve functions or  derivatives provides always a less efficient estimate (see Theorem \ref{GM-Repeated} and Remark \ref{R2}).
For this reason, our theoretical results may have a relevant  impact in practice.

More in detail, in analogy with the Riesz representation theorem we can find a representative function in $H^1$ which incorporates the information provided by a curve function and a derivative (which belong to $L^2$). Hence, from the two samples of reconstructed  functions and derivatives  just one sample of representatives  is obtained and 
we use this sample of representatives to estimate the functional parameters.
\rosso{Once this method is given, the consequent theoretical results may appear as a
straightforward extension of the well-known classical ones; their
proof, however, requires much more technical effort and it is not a straightforward extension at all.

The OLS estimator (based on both curves and derivatives through their Riesz representatives in $H^1$) is provided and some practical considerations are drawn. In general, the OLS estimator is not a BLUE, because of the possible correlation between curves and derivatives. Therefore, a different representation of the data is provided (which takes into into account this correlation) and then}
a new version of the Gauss-Markov theorem is  proved
in the proper infinite-dimensional space ($H^1$), showing that our 
sample of representatives carries all the relevant information on the parameters.
\rosso{More in detail, we propose an unbiased estimator which is linear with respect to the new sample of representatives and which minimizes a suitable covariance matrix (called global variance). This estimator is denoted $H^1$-functional SBLUE.

A simulation study shows numerically the superiority of the $H^1$-functional SBLUE with respect to both the OLS estimators based only on curves or derivatives. This suggests that both sources of information should be used jointly, when available. A rough way of considering information on both curves and derivatives is to make a convex combination of the two OLS estimators. However, simulations show that the $H^1$-functional SBLUE is more efficient, as expected.}

The paper is organized as follows. Section \ref{sect:model} describes the model and proposes the OLS estimator obtained from the Riesz representation of the data. \rosso{Section \ref{sect:practice} explains some considerations which are fundamental from a practical point of view. Section \ref{sec:SBLUE} presents the construction of the functional strong BLUE. Finally, Section \ref{sect:simul}  is devoted to the simulation study.} Section \ref{sect:discuss} is a summary together with some final remarks.
Some additional results and the proofs of theorems are deferred to \ref{appendix}.

\section{\rosso{Model description and Riesz representation}}\label{sect:model}

Let us consider a regression model where the
response $y$ is a random function which depends linearly on a vectorial (or scalar) known variable $\bf x$ through a functional coefficient, which needs to be estimated.
In particular, we assume that \rosso{there are $n$ units (subjects or clusters), and  $r\ge 1$ observations per unit at a condition $\mathbf{x}_i$  ($i=1, \dots, n$). Note that $\mathbf{x}_1, \dots \mathbf{x}_n$ are not necessarily different.
In this context of repeated measurements,} we consider
 the following random effect model:
\begin{equation}\label{eq:model2}
\begin{aligned}
{y}_{ij}(t) & = \mathbf{f}(\mathbf{x}_i)^T \boldsymbol{\beta}(t) + {\alpha}_i(t) + {\varepsilon}_{ij}(t)
\end{aligned} \qquad i= 1,\ldots,n;\  j=1,\ldots,r,
\end{equation}
where: \rosso{$t$ belongs to a compact set $\tau\subseteq \mathbb R$}; ${y}_{ij}(t)$ denotes the response curve of the $j$-th observation at the $i$-th experiment;
$\mathbf{f}(\mathbf{x}_i)$ is a $p$-dimensional vector of known functions;  $\boldsymbol{\beta}(t) $
is an unknown $p$-dimensional functional vector;
${\alpha}_i(t)$ is a zero-mean process which denotes the random effect due to the $i$-th experiment and takes into account the correlation among the $r$ repetitions; ${\varepsilon}_{ij}(t)$ is a zero-mean error process.

\rosso{Let us note that we are interested in precise estimation of the fixed effects $\boldsymbol{\beta}(t)$; herein the random effects are nuisance parameters.}

An example for the model \eqref{eq:model2}  \rosso{can be found in \citet{Shen04}, where an ergonomic problem is considered (in this case there are $n$ clusters of observations for the same individual)}; if $r=1$ this model reduces to the functional response model described, for instance, in \citet{HorKok:HK12}.

In a real world setting, the functions ${y}_{ij}(t)$ are not directly observed. By a smoothing procedure from the original data, the investigator can reconstruct both the functions and their first derivatives, obtaining ${y}_{ij}^{(f)}(t)$ and ${y}_{ij}^{(d)}(t)$, respectively. Hence we can assume that the model for the reconstructed functional data is
\begin{equation}\label{eq:rMeas}
\left\{
\begin{aligned}
{y}_{ij}^{(f)}(t) & = \mathbf{f}(\mathbf{x}_i)^T {\bbeta}(t) + {\alpha}_i^{(f)}(t) + {\varepsilon}_{ij}^{(f)}(t)\\
{y}_{ij}^{(d)}(t) & = \mathbf{f}(\mathbf{x}_i)^T {\bbeta}'(t) + {\alpha}_i^{(d)}(t) + {\varepsilon}_{ij}^{(d)}(t)
\end{aligned} \qquad i= 1,\ldots,n;\  j=1,\ldots,r,
\right.
\end{equation}
\rosso{where
\begin{enumerate}
\item \label{hp:1}
the $n$ couples $({\alpha}_i^{(f)}(t),{\alpha}_i^{(d)}(t))$ are independent and identically distributed bivariate vectors of zero-mean processes such that
$E(\|{\alpha}_i^{(f)}(t)\|^2_{L^2(\tau)}+\|{\alpha}_i^{(d)}(t)\|^2_{L^2(\tau)})<\infty$, that is, $({\alpha}_i^{(f)}(t),{\alpha}_i^{(d)}(t)) \in L^2(\Omega; \mathbf{L}^2 )$, where $\mathbf{L}^2 = L^2(\tau) \times L^2(\tau) $;
\item \label{hp:2}
the $n\times r$ couples $({\varepsilon}_{ij}^{(f)}(t),{\varepsilon}_{ij}^{(d)}(t))$ are independent and identically distributed bivariate vectors of zero mean processes processes, with  $E(\|{\varepsilon}_{ij}^{(f)}(t)\|^2_{L^2}+\|{\varepsilon}_{ij}^{(d)}(t)\|^2_{L^2})<\infty$.
\end{enumerate}
As a consequence of the above assumptions: the data ${y}_{ij}^{(f)}(t)$ and ${y}_{ij}^{(d)}(t)$ can be correlated; the couples $({y}_{ij}^{(f)}(t), {y}_{ij}^{(d)}(t))$ and $({y}_{kl}^{(f)}(t), {y}_{kl}^{(d)}(t))$ are independent whenever $i\neq k$. The possible correlation between $({y}_{ij}^{(f)}(t), {y}_{ij}^{(d)}(t))$ and $({y}_{il}^{(f)}(t), {y}_{il}^{(d)}(t))$ is due to the common random effect $({\alpha}_i^{(f)}(t),{\alpha}_i^{(d)}(t))$. } 

Note that the investigator might reconstruct each function ${y}_{ij}^{(f)}(t)$ and its derivative ${y}_{ij}^{(d)}(t)$ separately. In this case, the right-hand
term of the second equation in \eqref{eq:rMeas} is not the derivative of the right-hand term of the first equation. The particular case
when ${y}_{ij}^{(d)}(t)$ is obtained by differentiation ${y}_{ij}^{(f)}(t)$ is the most simple situation in model \eqref{eq:rMeas}.

Let $\mathbf{B}(t)$ be an estimator of $\bbeta (t)$, formed by $p$ random functions in the Sobolev space $H^1
$.
Recall that a function $g(t)$ is in $H^1$ if $g(t)$ and its derivative $g'(t)$ belong to $L^2$. Moreover, $H^1$ is a Hilbert
space with inner product
\begin{equation}
\label{eq:ip}
\begin{aligned}
\langle g_1(t), g_2(t) \rangle_{H^1} & = \langle g_1(t), g_2(t) \rangle_{L^2} + \langle g'_1(t), g'_2(t) \rangle_{L^2} 
\\
& =\rosso{\langle (g_1(t), g'_1(t)) , (g_2(t), g'_2(t)) \rangle_{\mathbf{L}^2}} \\
& = \int g_1(t) g_2(t) dt + \int g'_1(t) g'_2(t) dt , \qquad
g_1(t),g_2(t)\in H^1.
\end{aligned}
\end{equation}

\begin{definition}\label{defi:sigma}
We define the $H^1$-\emph{\rosso{global} covariance matrix} $\Sigma_{\mathbf{B}}$ of an unbiased estimator $\mathbf{B} (t)$  as the $p\times p$ matrix whose
$(l_1,{l_2})$-th element is
 \begin{equation}\label{eq:sigma}
 E \langle {B}_{l_1}(t)-\beta_{l_1}(t),{B}_{{l_2}}(t)-\beta_{{l_2}}(t)\rangle_{H^1}. 
 \end{equation}
\end{definition}
This global notion of covariance has been used also in \citet[Definition~2]{Menafoglio13}, in the context of predicting georeferenced functional data.
These authors have found a BLUE estimator for the drift of their underlying process, which can be seen as an example of the results given in this paper.

Given a couple $({y}^{(f)}(t),{y}^{(d)}(t)) \in L^2\times L^2$, it may be defined a linear continuous operator on $H^1$
as follows
\begin{equation*}
\phi(h)
= \langle {y}^{(f)} , h \rangle_{L^2} + \langle {y}^{(d)} , h' \rangle_{L^2}
\rosso{= \langle ({y}^{(f)} , {y}^{(d)}) \,,\, (h,h') \rangle_{\mathbf{L}^2}}, \qquad \forall h\in H^1.
\end{equation*}
From the  Riesz representation theorem, there exists a unique ${\tilde{y}} \in H^1$ such that
\begin{equation}\label{eq:rGM0}
\langle \tilde{y} , h \rangle_{H^1} 
= \langle {y}^{(f)} , h \rangle_{L^2} + \langle {y}^{(d)} , h' \rangle_{L^2}, \qquad \forall h\in H^1.
\end{equation}
\begin{definition}
\label{DEf:def3}
The unique element $\tilde{{y}} \in H^1$ defined in \eqref{eq:rGM0} is called the \emph{Riesz representative} of the couple $({y}^{(f)}(t),{y}^{(d)}(t)) \in \mathbf{L}^2$.
\end{definition}
This definition will be useful to provide a nice expression for the functional OLS estimator ${\Hbbeta}(t)$. Actually the Riesz
representative synthesizes, in some sense, in $H^1$ the information of both ${y}^{(f)}(t)$  and ${y}^{(d)}(t) $.

\rosso{Note that, since
	\[
\langle ({y}^{(f)} , {y}^{(d)}) - (\tilde{{y}},\tilde{{y}}') \,,\, (h,h') \rangle_{\mathbf{L}^2} =0,\qquad \forall h\in H^1
	\]
the Riesz
representative $(\tilde{{y}},\tilde{{y}}')$ may be seen as the projection of $({y}^{(f)} , {y}^{(d)})\in {\mathbf{L}^2}$ onto the immersion of $H^1$ in ${\mathbf{L}^2}$, a linear closed subspace.}

The \emph{functional OLS estimator} for the model \eqref{eq:rMeas} is
\begin{align*}
{\Hbbeta} (t) & = \arg\min_{\bbeta(t)} \Big(
\sum_{j=1}^r \sum_{i=1}^n \|
{y}_{ij}^{(f)}(t) - \mathbf{f}(\mathbf{x}_i)^T {\bbeta}(t)
\|^2_{L^2} +
\sum_{j=1}^r \sum_{i=1}^n \|
{y}_{ij}^{(d)}(t) - \mathbf{f}(\mathbf{x}_i)^T {\bbeta}'(t)
\|^2_{L^2}
\Big)
\\
& = \arg\min_{\bbeta(t)}
\sum_{j=1}^r \sum_{i=1}^n \Big( \|
{y}_{ij}^{(f)}(t) - \mathbf{f}(\mathbf{x}_i)^T {\bbeta}(t)
\|^2_{L^2} + \|
{y}_{ij}^{(d)}(t) - \mathbf{f}(\mathbf{x}_i)^T {\bbeta}'(t)
\|^2_{L^2}
\Big)
\end{align*}
The quantity
\[
\|
{y}_{ij}^{(f)}(t) - \mathbf{f}(\mathbf{x}_i)^T {\bbeta}(t)
\|^2_{L^2} +\|
{y}_{ij}^{(d)}(t) - \mathbf{f}(\mathbf{x}_i)^T {\bbeta}'(t)
\|^2_{L^2}
\]
resembles
\[
\|
{y}_{ij}(t) - \mathbf{f}(\mathbf{x}_i)^T {\bbeta}(t)
\|^2_{H^1} ,
\]
because ${y}_{ij}^{(f)}(t)$ and ${y}_{ij}^{(d)}(t)$ reconstruct ${y}_{ij}(t)$  and its derivative function, respectively. The functional OLS estimator ${\Hbbeta} (t)$ minimizes, in this sense, the sum of the $H^1$-norm of the unobservable residuals ${y}_{ij}(t) - \mathbf{f}(\mathbf{x}_i)^T {\bbeta}(t)$.
\begin{theorem}\label{thm:Riesz_representation}
Given model in \eqref{eq:rMeas},
\begin{itemize}
\item[a)] the functional OLS estimator ${\Hbbeta} (t)$ can be computed by
\begin{equation}\label{eq:Hbbeta}
{\Hbbeta} (t) = (F^TF)^{-1} F^T \bar{\mathbf{y}}(t),
\end{equation}
where $\bar{\mathbf{y}}(t) = (\bar{y}_1 (t), \ldots ,\bar{y}_n (t))\rosso{^T}$ is a vector, whose component $i$-th is the mean of the Riesz representatives of the replications:
\[
\bar{y}_i (t) = \frac{\sum_{j=1}^r \tilde{y}_{ij}(t)}{r},
\]
and $F = [\mathbf{f}(\mathbf{x}_1) ,\ldots, \mathbf{f}(\mathbf{x}_n)]^T$ is the $n\times p$ design matrix.
\item[b)] The estimator ${\Hbbeta} (t)$ is unbiased and its \rosso{global} covariance matrix is 
$\rosso{\sigma^2}(F^TF)^{-1}$.
\end{itemize}
\end{theorem}

\begin{remark}\label{rem:OLSlambda}
	The previous results may be generalized to other Sobolev spaces. The extension to $H^m$, $m\geq 2$,
	is straightforward.
	Moreover, in Bayesian context, the investigator might have a different \emph{a priori} consideration of ${y}_{ij}^{(f)}(t)$ and ${y}_{ij}^{(d)}(t)$.
	Thus, different weights may be used for curves and derivatives, and the inner product given in 
	\eqref{eq:ip} may be extended to
	\[
	\langle g_1(t), g_2(t) \rangle_{H} = \lambda \int_\tau g_1(t) g_2(t) dt + (1-\lambda) \int_\tau g'_1(t) g'_2(t) dt,
	\qquad \lambda \in [0,1].
	\]
Let $ {\Hbbeta}_\lambda (t) $  be the OLS estimator obtained by using this last inner product. Note that, 
for $\lambda= \tfrac{1}{2}$, we obtain ${\Hbbeta}_{\tfrac{1}{2}} (t) = {\Hbbeta} (t) $ defined in
 Theorem~\ref{thm:Riesz_representation}.
The behavior of the $ {\Hbbeta}_\lambda (t) $   is explored in Section~\ref{sect:simul} for different choices
of $\lambda$.
\end{remark}

\section{Practical considerations}\label{sect:practice}
In a real world context, we work with a finite dimensional subspace $\mathcal{S}$ of $H^1$. Let $S=\{w_1(t), \ldots, w_N(t) \}$ be a base of
$\mathcal{S}$. Without loss of generality,
we may assume that $\langle w_h(t), w_k(t)\rangle_{H^1} = \delta_{h}^k$,
where
\[
\delta_{h}^k=
\begin{cases}
1 & \text{if }h=k; \\
0 & \text{if }h\neq k;
\end{cases}
\]
is the Kronecker delta symbol, since a Gram-Schmidt orthonormalization
procedure may be always applied.
More precisely,
given any base
$\tilde{S}=\{\tilde{w}_1(t), \ldots, \tilde{w}_N(t) \}$ in $H_1$, the corresponding orthonormal base is given by:

for $k= 1$, define $w_1(t) = \frac{\tilde{w}_1(t)}{\|\tilde{w}_1(t)\|_{H^1}}$,

for $k\geq 2$, let
\(
\hat{w}_k(t) = \tilde{w}_k(t) - \sum_{h=1}^{n-1} \langle \tilde{w}_k(t) , {w}_h(t) \rangle_{H^1} {w}_h(t) , \) and \(
w_k(t) = \frac{\hat{w}_k(t)}{\|\hat{w}_k(t)\|_{H^1}}.
\)

With this orthonormalized base, the projection $\tilde{y}(t)_{\mathcal{S}}$ on $\mathcal{S}$ of the Riesz representative $\tilde{y}(t)$
of the couple $({y}^{(f)}(t),{y}^{(d)}(t)) $ is given by
\begin{equation}\label{eq:y_on_S}
\begin{aligned}
\tilde{y}(t)_{\mathcal{S}} &= \sum_{k=1}^N \langle \tilde{y}(t) , {w}_k(t) \rangle_{H^1} \cdot
{w}_k(t) \\
&= \sum_{k=1}^N \Big(\langle {y}^{(f)}(t) , {w}_k(t) \rangle_{L^2} + \langle {y}^{(d)}(t) , {w}'_k(t) \rangle_{L^2}\Big)
{w}_k(t) ,
\end{aligned}
\end{equation}
where the last equality comes from the definition \eqref{eq:rGM0} of the Riesz representative. Now, if $\mathbf{m}_{l} = ({m}_{l,1}, \ldots , {m}_{l,n})^T$ is the ${l}$-th row of $(F^TF)^{-1}F^T$, then
\begin{align*}
\langle\hat{\bbeta}_{l}(t) , {w}_k(t) \rangle_{H^1}
& = \sum_{i=1}^n
\langle {m}_{l,i} \bar{{y}_i}(t) , {w}_k(t) \rangle_{H^1}
\\
& =
\sum_{i=1}^n
{m}_{l,i} \langle \bar{{y}_i}(t) , {w}_k(t) \rangle_{H^1}
, \qquad \text{for any }k=1,\ldots, N, \\
{\hat{\bbeta}_{l}}(t)_{\mathcal{S}}
& = \mathbf{m}_{l}^T \bar{\mathbf{y}}(t)_{\mathcal{S}},
\end{align*}
hence ${\Hbbeta}(t)_{\mathcal{S}} = (F^TF)^{-1} F^T \bar{\mathbf{y}}(t)_{\mathcal{S}}$.

\bigskip

Let us note that, even if the Riesz representative \eqref{eq:rGM0} is implicitly defined, its projection on $\mathcal{S}$
can be easily computed by \eqref{eq:y_on_S}. From a practical point of view, the statistician can work with the data
$( {y}_{ij}^{(f)}(t), {y}_{ij}^{(d)}(t))$ projected on a finite linear subspace $\mathcal{S}$ and the corresponding OLS estimator
$\Hbbeta(t)_{\mathcal{S}}$ is the projection on $\mathcal{S}$ of the
OLS estimator $\Hbbeta(t)$ given in
Section~\ref{sect:model}.

It is straightforward to prove that the estimator \eqref{eq:Hbbeta} becomes
\[
\Hbbeta (t) = (F^TF)^{-1}F^T {\mathbf{y}}^{(f)}(t),
\]
in two cases: when we do not take into consideration $y^{(d)}$, or when $y^{(d)} = (y^{(f)})'$.
Up to our knowledge,
this is the most common situation considered in the literature (see \citet[Chapt. 13]{Ramsay:Silverman05}).
\rosso{However, from the simulation study of Section \ref{sect:simul}, the OLS estimator $\Hbbeta$ is less efficient when it is based only on $y^{(f)}$.}

\section{Strong $H^1$-BLUE in functional linear models}\label{sec:SBLUE}

Let $\mathbf{B}(t)=\mathbf{C}({\mathbf{y}}^{(f)} (t),{\mathbf{y}}^{(d)} (t))$, where
$\mathbf{C}: \mathcal{R}\subseteq (\mathbf{L}^2)^{n r}\to (H^1)^p$ is a linear closed operator; in this case $\mathbf{B}(t)$ is called a {\it linear estimator}.
The domain of $C$, denoted by $\mathcal{R}$, will be defined in \eqref{defn:DomR}. 
Theorem~\ref{cor:yR-eR} will ensure that the dataset $({\mathbf{y}}^{(f)} (t),{\mathbf{y}}^{(d)} (t))$ 
is contained in $\mathcal{R}$.

\begin{definition}
	In analogy with classical settings, we define the $H^1$-functional best linear unbiased estimator ($H^1$-BLUE) as the estimator with minimal
	(in the sense of Loewner Partial Order\footnote{Given two symmetric matrices $A$ and $B$, $A\ge B$ 
	in Loewner Partial Order if $A-B$ is positive definite.}) $H^1$-\rosso{global} covariance matrix \eqref{eq:sigma}, in the class of the linear unbiased estimators $\mathbf{B}(t)$ of $\bbeta (t)$.
\end{definition}

\rosso{From the definition of Loewner Partial Order, a $H^1$-BLUE } 
minimizes the quantity
\[
E\Big( \Big\langle \sum_{i=1}^p \alpha_i \big( {B}_{i}(t)-\beta_{i}(t) \big) , \sum_{i=1}^p
\alpha_i \big( {B}_{i}(t)-\beta_{i}(t) \big) \Big\rangle_{H^1}\Big)
\]
for any choice of $(\alpha_1, \ldots, \alpha_p)$, in the class of the linear unbiased estimators $\mathbf{B}(t)$ of $\bbeta (t)$.
In other words, the $H^1$-BLUE minimizes the $H^1$-\rosso{global} variance of any linear combination of its components. 
A stronger request is the following.
\begin{definition}
	We define the $H^1$-strong functional best linear unbiased estimator ($H^1$-{\color{red}S}BLUE) as the estimator with minimal \rosso{global}
	variance,
	\[
	E\Big( \Big\langle {\mathrm{O}}( \rosso{\mathbf{B}}(t)-\rosso{\bbeta}(t) ) , {{\mathrm{O}}} ( \rosso{\mathbf{B}}(t)-\rosso{\bbeta}(t) ) \Big\rangle_{H^1}\Big)
	\]
	for any choice of a (sufficiently regular) continuous  linear operator ${{\mathrm{O}}}:(H^1)^p \to H^1$,
	in the class of the linear unbiased estimators $\mathbf{B}(t)$ of $\bbeta (t)$.
\end{definition}

\subsection{\rosso{$H^1_{\mathbf{R}}$-representation on the Hilbert space $\mathbf{L}^2_{\mathbf{R}}$}}

Recall that, for any given $(i,j)$, the couple $({\alpha}_i^{(f)}(t)+ {\varepsilon}_{ij}^{(f)}(t),{\alpha}_i^{(d)}(t)+ {\varepsilon}_{ij}^{(d)}(t))$
is a process with values in $\mathbf{L}^2 = L^2(\tau) \times L^2(\tau)$.
Let $\mathbf{R}(s,t) = \sum_k \lambda_k \boldsymbol\Psi_k(s) \boldsymbol\Psi_k(t)^T$ be the spectral representation of the covariance matrix 
of the process 
\begin{equation}\label{eq:E_i}
\mathbf{e}_i^T(t) = ({e}_i^{(f)}(t)\,,\,{e}_i^{(d)}(t)) = 
\frac{1}{r}\sum_{j=1}^r({\alpha}_i^{(f)}(t)+ {\varepsilon}_{ij}^{(f)}(t) \,,\, {\alpha}_i^{(d)}(t)+ {\varepsilon}_{ij}^{(d)}(t)),
\qquad i=1,\ldots,n
\end{equation}
which means $\lambda_k>0$, $\sum_k \lambda_k<\infty$ and the sequence $\{ \boldsymbol\Psi_k(t),k=1,2,\ldots\}$ are orthonormal bivariate vectors in $\mathbf{L}^2$. Without loss of generality 
assume that the $\mathbf{L}^2$-closure of the
linear span of $\{ \boldsymbol\Psi_k(t),k=1,2,\ldots\}$ includes $H^1$ (see Remark~\ref{rem:Psik}): 
$\overline{{\mathbf{L}^2} \cap \text{span} \{ \boldsymbol\Psi_k(t),k=1,2,\ldots\}} \supseteq H^1$.
Note that $\mathbf{R}(s,t)$, the covariance matrix 
of the process 
$\mathbf{e}_i(t) $,
does not depend on $i$. 
From Karhunen--Lo\`eve Theorem (see, e.g., \citet{Perrin13}), there exists an array of zero-mean unit variance random variables
$\{e_{i,k}; i=1,\ldots,n;k=1,2,\ldots\}$ such that
\begin{equation}
\mathbf{e}_i(t) = 
\sum_k \sqrt{\lambda_k} e_{i,k} \boldsymbol\Psi_k(t).
\label{eq:noise}
\end{equation}
The linearity of the covariance operator with respect to the first process, together with the symmetry in $j$ given 
in the hypothesis (\ref{hp:1}) and (\ref{hp:2}), ensures that
\begin{equation}\label{eq:corr_eE}
E \Big[ \big({\alpha}_{i}^{(f)}(s)+ {\varepsilon}_{ij}^{(f)}(s) \,,\, {\alpha}_{i}^{(d)}(s)+ {\varepsilon}_{ij}^{(d)}(s)\big)^T 
\cdot \mathbf{e}^T_i(t) \Big]
= \mathbf{R}(s,t) = \sum_k \lambda_k \boldsymbol\Psi_k(s) \boldsymbol\Psi_k(t)^T.
\end{equation}
Now, for $i=1,\ldots,n;j=1,\ldots,r;k=1,2,\ldots$, let 
$$
X_{ij,k} = \Big\langle \boldsymbol\Psi_{k} , 
\big({\alpha}_{i}^{(f)}+ {\varepsilon}_{ij}^{(f)} \,,\, {\alpha}_{i}^{(d)}+ {\varepsilon}_{ij}^{(d)}\big)^T \Big\rangle_{\mathbf{L}^2},
$$
and hence
\[
\big({\alpha}_{i}^{(f)}(s)+ {\varepsilon}_{ij}^{(f)}(s) \,,\, {\alpha}_{i}^{(d)}(s)+ {\varepsilon}_{ij}^{(d)}(s)\big)^T= 
\sum_k X_{ij,k} \boldsymbol\Psi_k(s), \qquad \frac{1}{r} \sum_{j=1}^r X_{ij,k} = \sqrt{\lambda_k}e_{i,k}.
\]
The independence assumptions in the hypothesis (\ref{hp:1}) and (\ref{hp:2}) ensures that the joint law of
the processes $({\alpha}_{i_1}^{(f)}+ {\varepsilon}_{i_1j}^{(f)} \,,\,{\alpha}_{i_1}^{(d)}+ {\varepsilon}_{i_1j}^{(d)})$
and $\mathbf{e}_{i_2}$ does not depend on $j$, hence
\[
E (X_{i_11,k_1}   \sqrt{\lambda_{k_2}}e_{i_2,k_2} ) = 
E (X_{i_12,k_1}   \sqrt{\lambda_{k_2}}e_{i_2,k_2} ) = \cdots = E (X_{i_1r,k_1}   \sqrt{\lambda_{k_2}}e_{i_2,k_2} ). 
\]
From \eqref{eq:corr_eE}, the linearity of the expectation ensures that
\begin{equation}\label{eq:rGMprod}
\delta_{i_1}^{i_2} \delta_{k_1}^{k_2} \lambda_{k_1}  = 
E ( \sqrt{\lambda_{k_1}}e_{i_1,k_1}  \sqrt{\lambda_{k_2}}e_{i_2,k_2} ) = \sqrt{\lambda_{k_2}}
E (X_{i_1j,k_1}   e_{i_2,k_2} ), \qquad j=1,\ldots,r.
\end{equation}
Let us observe that the elements of $\overline{{\mathbf{L}^2} \cap \text{span} \{ \boldsymbol\Psi_k(t),k=1,2,\ldots\}}$ are the functions $\mathbf{a}$ such that
$\mathbf{a}=\sum_k  \langle \mathbf{a} , \boldsymbol\Psi_{k}\rangle_{\mathbf{L}^2} \cdot \boldsymbol\Psi_{k}$ and
$\| \mathbf{a} \|_{\mathbf{L}^2} ^2 = \sum_k  \langle \mathbf{a} , \boldsymbol\Psi_{k}\rangle_{\mathbf{L}^2}  ^2< \infty$. In the following definition a stronger condition is required.
\begin{definition}
Given the spectral representation 
of $\mathbf{R}(s,t)$, let
\begin{equation}\label{eq:L2Omega}
\mathbf{L}^2_{\mathbf{R}} = 
\Big\{ \mathbf{a} \in \overline{{\mathbf{L}^2} \cap \text{span} \{ \boldsymbol\Psi_k(t),k=1,2,\ldots\}}
\colon \sum_k 
\frac{ \langle \mathbf{a} , \boldsymbol\Psi_{k}\rangle_{\mathbf{L}^2}  ^2 }{\lambda_k} < \infty \Big\}
\end{equation}
be a new Hilbert space,
with inner product
\begin{equation}\label{eq:innnerOmega}
\langle \mathbf{a} , \mathbf{b} \rangle_{\mathbf{L}^2_{\mathbf{R}}} = \sum_k 
\frac{ \langle \mathbf{a} , \boldsymbol\Psi_{k}\rangle_{\mathbf{L}^2}  \langle 
\mathbf{b} , \boldsymbol\Psi_{k} \rangle_{\mathbf{L}^2} }{\lambda_k}.
\end{equation}
\end{definition}
Note that \(\| \cdot \|_{\mathbf{L}^2} \leq \tfrac{1}{\max(\lambda_k)}\| \cdot \|_{\mathbf{L}^2_{\mathbf{R}}}\).
An orthonormal base for $\mathbf{L}^2_{\mathbf{R}}$ is given by $(\boldsymbol\Phi_{k})_k$, where 
$\boldsymbol\Phi_{k} = \sqrt{\lambda_k}\boldsymbol\Psi_{k}$ for any $k$.

Consider now the following linear closed dense subset of $\mathbf{L}^2_{\mathbf{R}} $:
\[
K = \Big\{ \mathbf{b} \in {\mathbf{L}^2_{\mathbf{R}}} \colon \sum_k 
\frac{ \langle \boldsymbol\Psi_{k},\mathbf{b} \rangle_{\mathbf{L}^2}^2 }{\lambda_k^2} < \infty \Big\}.
\]
Observe that $\boldsymbol\Psi_{k} \in K$ for all $k$. If $K^*$ is the ${\mathbf{L}^2_{\mathbf{R}}}$-dual space of $K$,
the Gelfand triple $K \subset \mathbf{L}^2_{\mathbf{R}} \subset K^*$ 
implies that $\overline{{\mathbf{L}^2} \cap \text{span} \{ \boldsymbol\Psi_k(t),k=1,2,\ldots\}} \subseteq K^*$.

In analogy with the geometric interpretation of the Riesz representation, we construct the \emph{$H^1_{\mathbf{R}}$-representation} 
in the following way. For any element $\mathbf{b} \in \mathbf{L}^2_{\mathbf{R}}$, we call \emph{$H^1_{\mathbf{R}}$-representative} its $\mathbf{L}^2_{\mathbf{R}}$-projection on $H^1$, and we denote it with the symbol $b^{(\mathbf{R})}$.
In particular, for any $k$, let ${\psi^{(\mathbf{R})}_k}(t)$ be the $H^1_{\mathbf{R}}$-representative of $\boldsymbol\Psi_{k} $, that is, the unique element in $H^1\cap \mathbf{L}^2_{\mathbf{R}}$ such that
\[
\langle ({\psi^{(\mathbf{R})}_k},{\psi^{(\mathbf{R})}_k}')^T , (g,g')^T 
\rangle_{\mathbf{L}^2_{\mathbf{R}}} = 
\langle \boldsymbol\Psi_{k} , (g,g')^T \rangle_{\mathbf{L}^2_{\mathbf{R}}} = 
\frac{\langle \boldsymbol\Psi_{k} , (g,g')^T \rangle_{\mathbf{L}^2}}{\lambda_k} , \qquad \forall g \in H^1\cap \mathbf{L}^2_{\mathbf{R}}.
\]
Note that the $H^1_{\mathbf{R}}$-representatives of the orthonormal system $(\boldsymbol\Phi_{k})_k$ of $\mathbf{L}^2_{\mathbf{R}}$ are given by 
${\phi^{(\mathbf{R})}_k}(t) =\sqrt{\lambda_k}{\psi^{(\mathbf{R})}_k}(t)$, where, by definition of projection,
\begin{equation}\label{eq:projH1Rbase}
\| {\phi^{(\mathbf{R})}_k}(t)\|_{\mathbf{H}^1_{\mathbf{R}}} =
\| ({\phi^{(\mathbf{R})}_k}(t),{\phi^{(\mathbf{R})}_k}'(t))^T \|_{\mathbf{L}^2_{\mathbf{R}}} \leq 
\| \boldsymbol{\Phi}_k(t) \|_{\mathbf{L}^2_{\mathbf{R}}} = 1.
\end{equation}
Moreover,
\begin{equation}\label{eq:tricks}
\langle ({\phi^{(\mathbf{R})}_h},{\phi^{(\mathbf{R})}_h}')^T, \boldsymbol\Phi_{k}\rangle_{\mathbf{L}^2_{\mathbf{R}}}
=
\langle ({\phi^{(\mathbf{R})}_h},{\phi^{(\mathbf{R})}_h}')^T , ({\phi^{(\mathbf{R})}_k},{\phi^{(\mathbf{R})}_k}') ^T\rangle_{\mathbf{L}^2_{\mathbf{R}}} 
=
\langle \boldsymbol\Phi_{h}, ({\phi^{(\mathbf{R})}_k},{\phi^{(\mathbf{R})}_k}') ^T\rangle_{\mathbf{L}^2_{\mathbf{R}}} ,
\end{equation}
and 
the $H^1_{\mathbf{R}}$-representation of any $\mathbf{b} \in \mathbf{L}^2_{\mathbf{R}}$ can be written as 
\begin{equation}\label{eq:H1RRreprL2R}
{b}^{(\mathbf{R})} = \sum_h \langle \mathbf{b} , \boldsymbol\Psi_{h}\rangle_{\mathbf{L}^2} {\psi^{(\mathbf{R})}_h}
= \sum_h \langle \mathbf{b} , \boldsymbol\Phi_{h}\rangle_{\mathbf{L}^2_{\mathbf{R}}} {\phi^{(\mathbf{R})}_h}.
\end{equation}
When $\mathbf{a} \in \overline{{\mathbf{L}^2} \cap \text{span} \{ \boldsymbol\Psi_k(t),k=1,2,\ldots\}}$, 
it is again possible to define formally its $H^1_{\mathbf{R}}$-representation 
in the following way:
\begin{equation}\label{eq:H1RRreprK}
{a}^{(\mathbf{R})}(t) = \sum_k \langle \mathbf{a} , \boldsymbol\Psi_{k}\rangle_{\mathbf{L}^2} {\psi^{(\mathbf{R})}_k}(t).
\end{equation}
In this case, if ${a}^{(\mathbf{R})}\in H^1$, an analogous of the standard projection can be obtained:
$ ({a}^{(\mathbf{R})},{{a}^{(\mathbf{R})}}')$ it is the unique element in $K^*$ of the form
$(a,a')$ with $a\in H^1$ such that
\[
\langle \mathbf{a}, (h,h')^T\rangle_{\mathbf{L}^2_{\mathbf{R}}} 
= \langle (a,a')^T , (h,h')^T\rangle_{\mathbf{L}^2_{\mathbf{R}}} , \qquad \forall (h,h')\in K.
\]
It will be useful to observe that, as a consequence, 
when $\mathbf{a}=(\mathbf{f}(\mathbf{x}_i)^T \boldsymbol{\beta} ,\mathbf{f}(\mathbf{x}_i)^T \boldsymbol{\beta} ')$, 
then its $H^1_{\mathbf{R}}$-representative is $\mathbf{f}(\mathbf{x}_i)^T \boldsymbol{\beta}$.
\begin{lemma}\label{lem:eRinH1}
Given $\mathbf{e}_i$ as in \eqref{eq:E_i},
its $H^1_{\mathbf{R}}$-representative 
\[
e_i^{(\mathbf{R})} = \sum_{k} \sqrt{\lambda_k}e_{i,k} {\psi^{(\mathbf{R})}_k} ,
\] 
belongs to $L^2(\Omega;{H^1})$, for any
$i=1,\ldots,n$.
\end{lemma}

The following theorem is a direct consequence of the previous results.
\begin{theorem}\label{cor:yR-eR}
The following equation holds in  $L^2(\Omega;{H^1})$:
\[
\bar{y}_i^{(\mathbf{R})}(t)(\omega) = \mathbf{f}(\mathbf{x}_i)^T \boldsymbol{\beta} (t) +
e_i^{(\mathbf{R})} (t)(\omega)
\qquad i=1,\ldots,n,
\]
where each 
\(
\bar{y}_i^{(\mathbf{R})} 
\) 
is the $H^1_{\mathbf{R}}$-representation 
of the mean $(\bar{\mathbf{y}}^{(f)}_i(t) ,\bar{\mathbf{y}}^{(d)}_i(t) )$ of the observations given in \eqref{eq:meanObs}.
As a consequence, $\bar{y}_i^{(\mathbf{R})}(t)$ belongs to $L^2(\Omega;{H^1})$, and hence 
$\bar{y}_i^{(\mathbf{R})}(\omega)\in H^1$ a.s.
\end{theorem}
We define 
\begin{equation}\label{defn:DomR}
\mathcal{R} = \{\mathbf{y}\in
\big(\overline{{\mathbf{L}^2} \cap \text{span} \{ \boldsymbol\Psi_k(t),k=1,2,\ldots\}}\big)^{n r} 
\colon y_i^{(\mathbf{R})} \in H^1,  i=1,\ldots,n\} .
\end{equation}

The vector $\bar{\mathbf{y}}^{(\mathbf{R})}(t) = \big(\bar{y}_1^{(\mathbf{R})} , \bar{y}_2^{(\mathbf{R})} , 
\ldots , \bar{y}_n^{(\mathbf{R})}\big)^T$
plays the r\^ole of the Riesz representative of Theorem~\ref{thm:Riesz_representation} in the 
following {\color{red}S}BLUE theorem.
\begin{theorem}\label{GM-Repeated}
	The functional estimator 
\begin{equation}\label{eq:HbbetaR}
{\HbbetaR} (t) = (F^TF)^{-1} F^T \bar{\mathbf{y}}^{(\mathbf{R})}(t),
\end{equation}
	for the model \eqref{eq:rMeas} is a $H^1$-functional
	{\color{red}S}%
	BLUE.	
\end{theorem}

\begin{remark}
\label{R2}
	From the proof of Theorem \ref{GM-Repeated} (see \ref{appendix}) we have that $\HbbetaR(t)$ is the best estimator among all the estimators $\mathbf{B}(t)=\mathbf{C}({\mathbf{y}}^{(f)} (t),{\mathbf{y}}^{(d)} (t))$ where
	$\mathbf{C}: {\mathcal{R}}\to (H^1)^p$ is any linear closed unbiased operator.
	Therefore, $\HbbetaR(t)$ is also better than the best linear unbiased estimators  based only on ${\mathbf{y}}^{(f)} (t)$ or ${\mathbf{y}}^{(d)} (t)$, since they are defined by some linear unbiased operator. 
	
\end{remark}

\begin{remark}\label{rem:Psik}
The assumption $\overline{{\mathbf{L}^2} \cap \text{span} \{ \boldsymbol\Psi_k(t),k=1,2,\ldots\}} \supseteq H^1$ 
ensures that the each component 
of the unknown $\bbeta (t) $ is in $\text{span} \{ \boldsymbol\Psi_k(t),k=1,2,\ldots\}$.
As a consequence, we have noted that the $H^1_{\mathbf{R}}$-representative of
$(\mathbf{f}(\mathbf{x}_i)^T \boldsymbol{\beta} ,\mathbf{f}(\mathbf{x}_i)^T \boldsymbol{\beta} ')$, 
is $\mathbf{f}(\mathbf{x}_i)^T \boldsymbol{\beta}$. 
If this assumption is not true, it may happen that $\beta_l \not\in \text{span} \{ \boldsymbol\Psi_k(t),k=1,2,\ldots\}$ for some $l=1,\ldots, p$, 
and then $\beta_l$ would have a nonzero projection on the orthogonal complement of 
$\text{span} \{ \boldsymbol\Psi_k(t),k=1,2,\ldots\} $. 
Since on the orthogonal complement we do not observe any noise,
this means that we would have a deterministic subproblem, that, without loss of generality, we can
ignore.
\end{remark}

\section{\rosso{Simulations}}\label{sect:simul}
In this section, it is explored, throughout a simulation study, when it is more convenient to use the whole information on both reconstructed functions and derivatives  with respect to the partial use of $y^{(f)}(t)$ (or $y^{(d)}(t)$). 
The idea is that using the whole information on curves and derivatives is much more convenient as 
the  dependence between $y^{(f)}(t)$ and $y^{(d)}(t)$ is smaller and their  spread is more comparable. 

\begin{figure}
\fbox{\parbox{.99\textwidth}{\center{\footnotesize{\textbf{Functions}}}\\ \includegraphics[width=.98\textwidth]{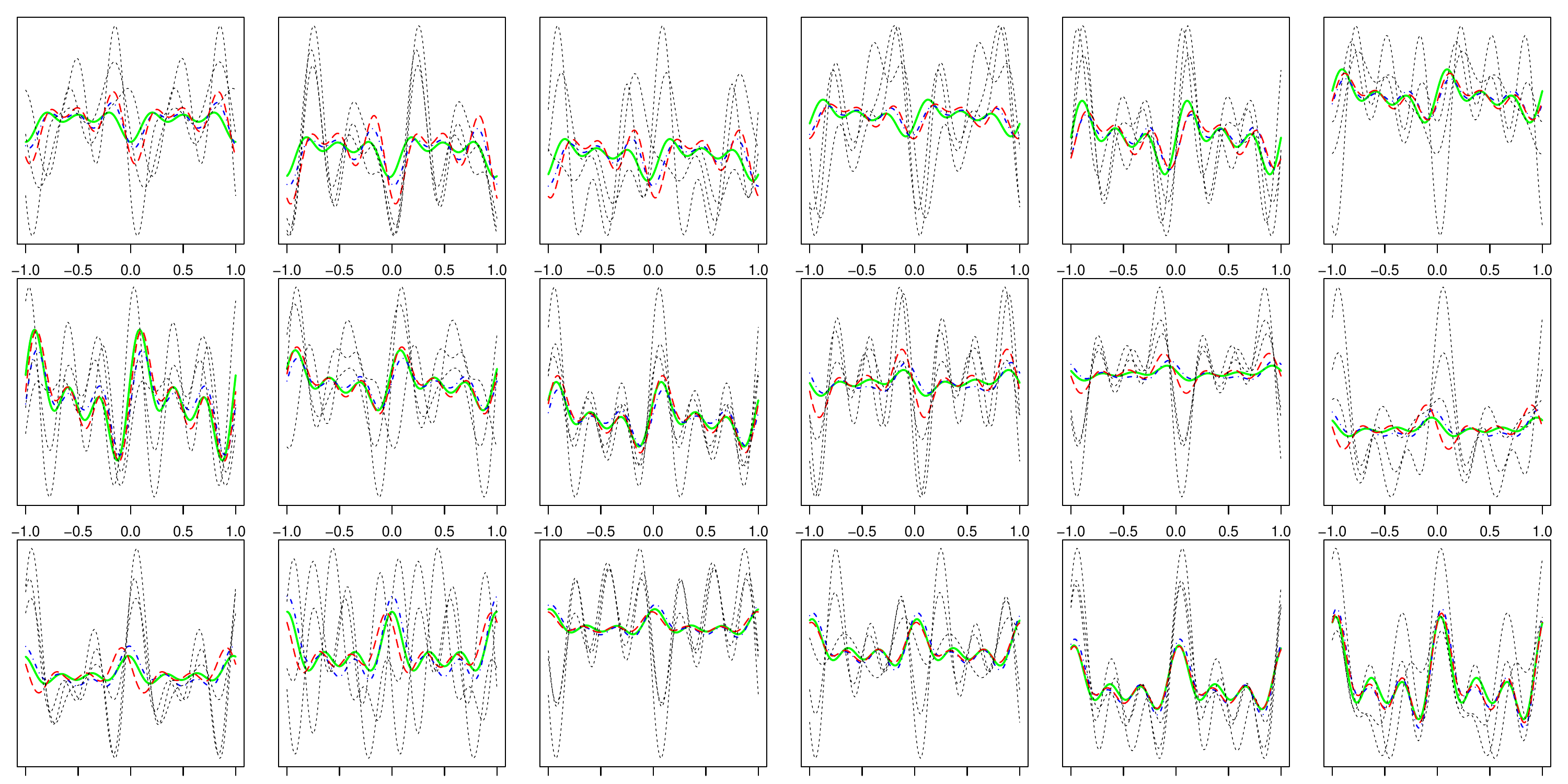}}}

\vspace{0.4cm}

\fbox{\parbox{.99\textwidth}{\center{\footnotesize{\textbf{Derivatives}}}\\ \includegraphics[width=.98\textwidth]{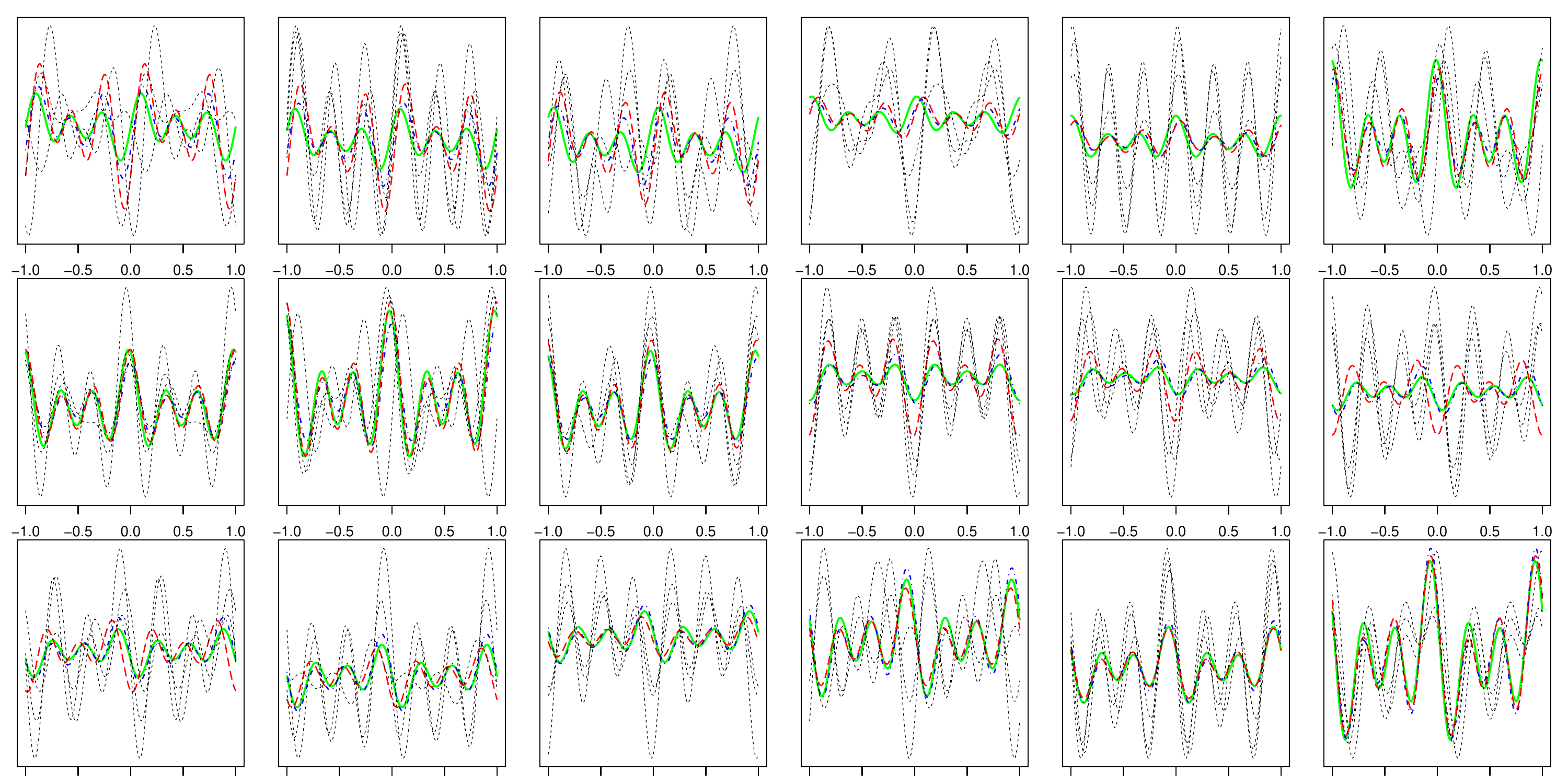}}}
\caption{Simulated data from model \eqref{eq:rMeas} and predicted curves. 
Black lines: simulated data of curves (top panel) and derivatives (bottom panel).
In each $i$-th box ($i=1,\ldots,18$) the $j=1, \ldots , 3$ replications are plotted.
Blue lines: predictions based on {\color{red}S}%
BLUE estimator.
Red lines: 
predictions based on OLS estimator.
Green lines: theoretical curves $\mathbf{f}(\mathbf{x}_i)^T \boldsymbol{\beta}(t)$ in top panel 
and $\mathbf{f}(\mathbf{x}_i)^T \boldsymbol{\beta}'(t)$ in bottom panel.}\label{fig:datasets}
\end{figure}

In this study, for each scenario listed below, $1000$ datasets are simulated from model \eqref{eq:rMeas} by a Montecarlo method, 
with $n=18$, $r=3$, $p=3$, 
\[
\bbeta (t) = 
\begin{pmatrix}
\sin (\pi t) + \sin (2 \pi t) + \sin (4 \pi t) \\
- \sin (\pi t) + \cos (\pi t) - \sin (2 \pi t) + \cos (2 \pi t)  - \sin (4 \pi t) + \cos (4 \pi t)  \\
+ \sin (\pi t) + \cos (\pi t) + \sin (2 \pi t) + \cos (2 \pi t) + \sin (4 \pi t) + \cos (4 \pi t) 
\end{pmatrix}, \qquad t \in (-1,1),
\]
and 
\[
F^T = \bigg[
\begin{smallmatrix}
 1 & 1 & 1 & 1 & 1 & 1 & 1 & 1 & 1 & 1 & 1 & 1 & 1 & 1 & 1 & 1 & 1 & 1 \\[1mm]
  0 & 0 & 0 & 0 & 0 & 0 & 0 & 0 & 0 & 1 & 1 & 1 & 1 & 1 & 1 & 1 & 1 & 1 \\[1mm]
 -1.00 & -0.75 & -0.50 & -0.25 & 0.00 & 0.25 & 0.50 & 0.75 & 1.00 & -1.00 & -0.75 & -0.50 & -0.25 & 0.00 & 0.25 & 0.50 & 0.75 & 1.00 
\end{smallmatrix}
\bigg].
\]
In what follows, we compare the following different estimators: the {\color{red}S}%
BLUE ${\HbbetaR} (t)$ (see Section~\ref{sec:SBLUE}), 
the OLS estimators 
${\Hbbeta}_\lambda (t) $ (see Remark~\ref{rem:OLSlambda}), and 
$\hat\bbeta^{(c)}_\lambda(t)=\lambda \hat\bbeta^{(f)}(t)+(1-\lambda)\hat\bbeta^{(d)}(t)$, where
$\hat\bbeta^{(f)}(t)$  is the OLS estimator based on $\mathbf{y}^{(f)}(t)$ and $\hat\bbeta^{(d)}(t)$ is 
the OLS estimator based on $\mathbf{y}^{(d)}(t)$, with $0\leq\lambda\leq 1$. 

Let us note that $\hat\bbeta^{(c)}_\lambda(t)$ is a  compound OLS estimator; it is a
rough way of taking into account both the sources of information on $\mathbf{y}^{(f)}(t)$ and $\mathbf{y}^{(d)}(t)$. 
Of course, setting $\lambda=0$ we ignore completely the information on the functions and $\hat\bbeta^{(c)}_0(t)=\hat\bbeta^{(d)}(t)= {\Hbbeta}_0(t)$, viceversa setting $\lambda=1$ means to ignore the information on the derivatives and thus $\hat\bbeta^{(c)}_1(t)=\hat\bbeta^{(f)}(t)= {\Hbbeta}_1 (t)$.

All the computations are developed using R package.

In Figure~\ref{fig:datasets} it is plotted: one dataset of curves and derivatives (black lines); 
the  regression functions
$\mathbf{f}(\mathbf{x}_i)^T \boldsymbol{\beta}(t)$ and $\mathbf{f}(\mathbf{x}_i)^T \boldsymbol{\beta}'(t)$ 
(green lines); the {\color{red}S}%
BLUE predictions
$\mathbf{f}(\mathbf{x}_i)^T {\HbbetaR} (t)$ and $\mathbf{f}(\mathbf{x}_i)^T {{\HbbetaR}}{}' (t)$  (blue lines);
the OLS predictions
$\mathbf{f}(\mathbf{x}_i)^T {\Hbbeta}_{\frac{1}{2}} (t) $ and $\mathbf{f}(\mathbf{x}_i)^T {{\Hbbeta}_{\frac{1}{2}} }{}' (t)$ 
 (red lines).

\subsection{Dependence between functions and derivatives}
We consider three different scenarios; we generate functional data $y_{ij}^{(f)}(t)$ and $y_{ij}^{(d)}(t)$ such that
%
\begin{enumerate}
\item
$({\alpha}_i^{(f)}(t) , {\varepsilon}_{ij}^{(f)}(t))$ is independent on
$({\alpha}_i^{(d)}(t) , {\varepsilon}_{ij}^{(d)}(t))$;
\item
$({\alpha}_i^{(f)}(t) , {\varepsilon}_{ij}^{(f)}(t))$ and
$({\alpha}_i^{(d)}(t) , {\varepsilon}_{ij}^{(d)}(t))$
are mildly dependent (the degree of dependence is randomly obtained); 
\item
$({\alpha}_i^{(f)}(t) , {\varepsilon}_{ij}^{(f)}(t))$ and
$({\alpha}_i^{(d)}(t) , {\varepsilon}_{ij}^{(d)}(t))$
are fully dependent:
$({\alpha}_i^{(d)}(t) , {\varepsilon}_{ij}^{(d)}(t)) = ({{\alpha}_i^{(f)}}'(t) , {{\varepsilon}_{ij}^{(f)}}'(t))$, and 
hence $y_{ij}^{(d)}(t)={y_{ij}^{(f)}}'(t)$.
\end{enumerate}  

The performance of the different estimators is evaluated by comparing the $H^1$-norm of the $p$-components of the
estimation errors. Figures~\ref{fig:MC1} depicts the Montecarlo distribution of the $H^1$-norm of the first component: 
$\| \hat{\bbeta}_{\lambda,1}(t) - {\bbeta}_1(t) \|_{H^1} $ for different values of 
$\lambda$ (red box-plot,  \eqref{eq:Hbbeta}), $\| \hat\bbeta^{(c)}_{\lambda,1}(t) - {\bbeta}_1(t) \|_{H^1} $ 
for different values of 
$\lambda$ (yellow box-plots) and 
$\| {\HbbetaR}_1(t) - {\bbeta}_1(t) \|_{H^1} $ (blue box-plot).
\begin{figure}
\begin{center}
\fbox{\includegraphics[width=.48\textwidth]{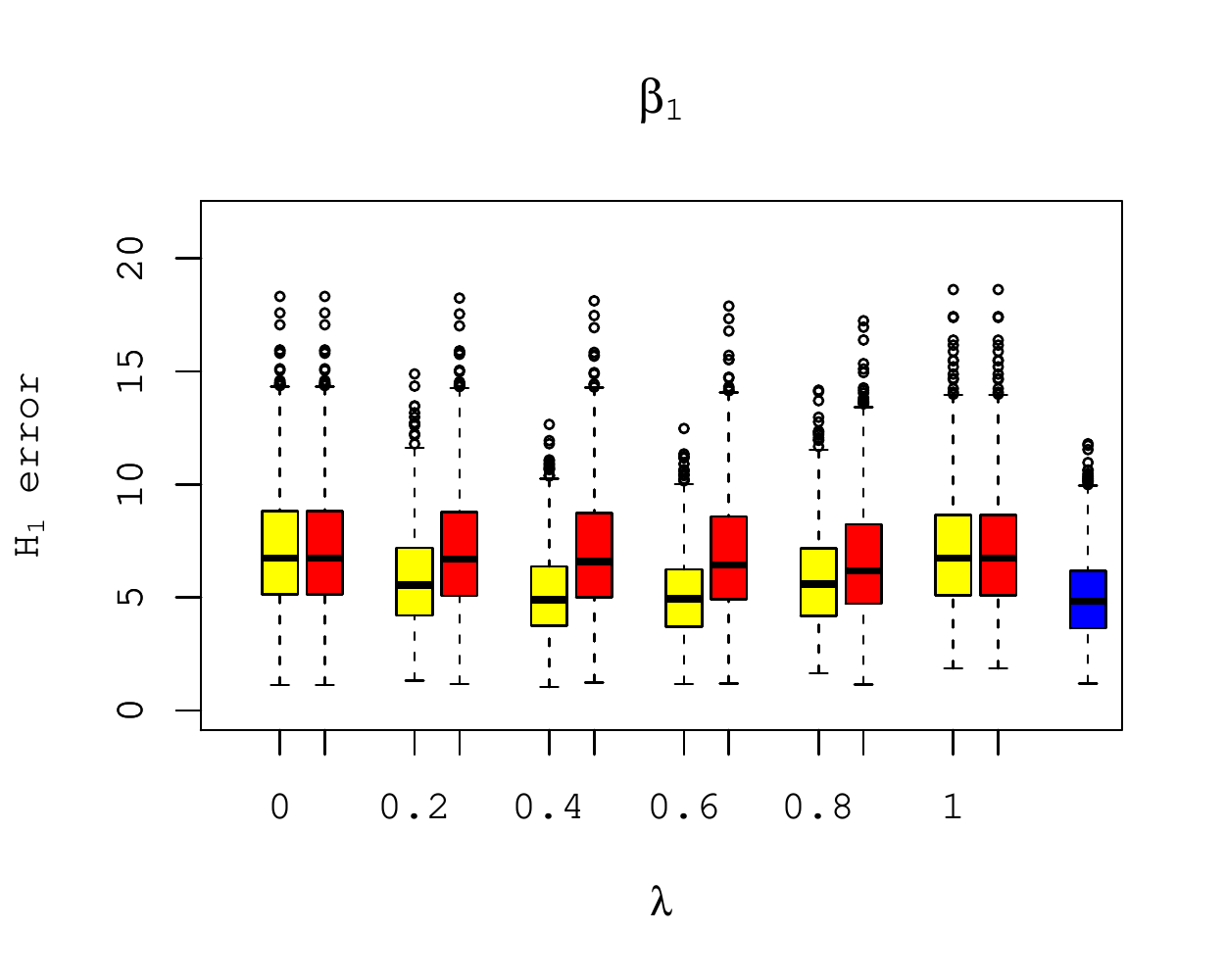}} 
\fbox{\includegraphics[width=.48\textwidth]{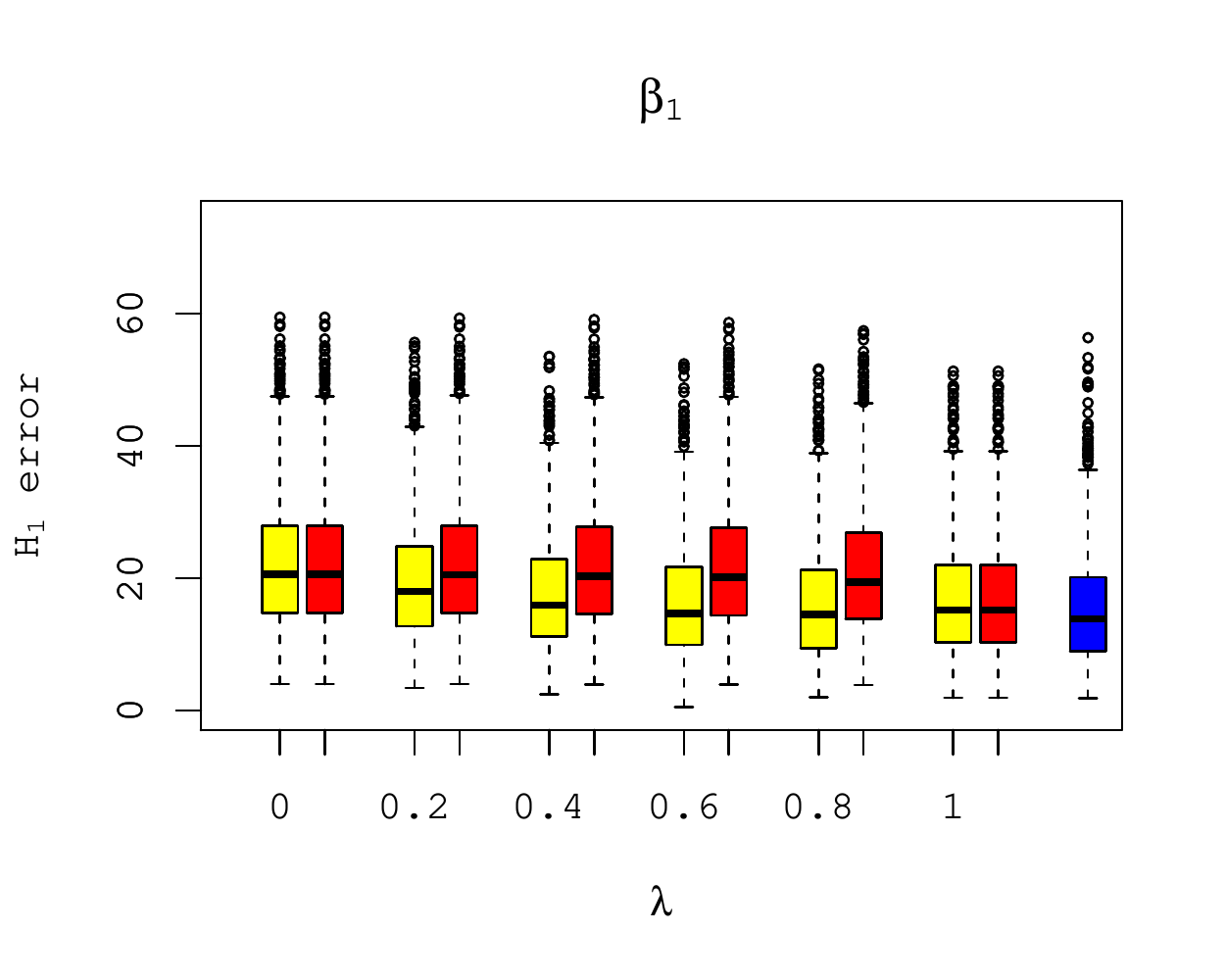}}
\fbox{\includegraphics[width=.48\textwidth]{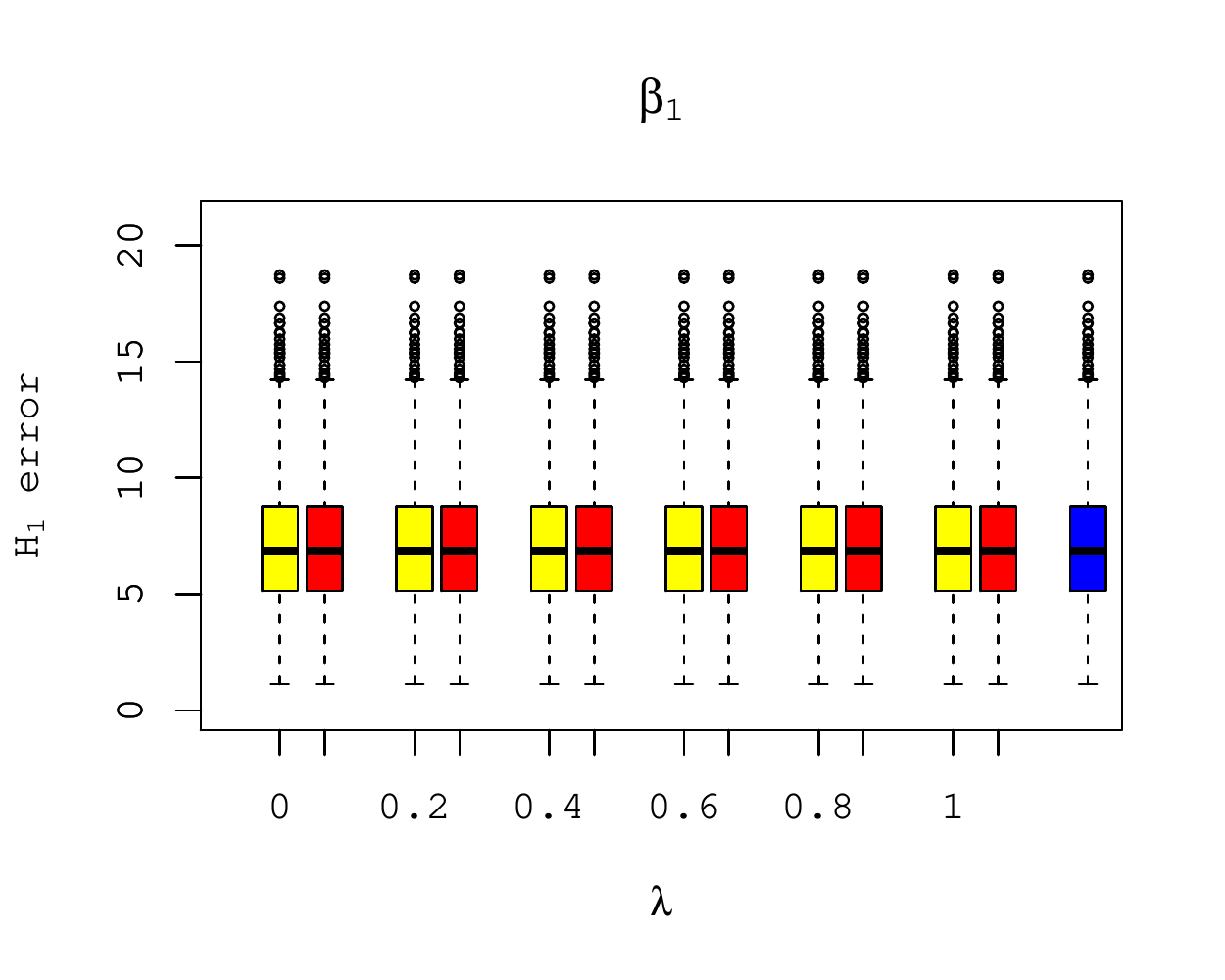}}
\end{center}
\caption{$H^1$ norm of the first components estimation errors, for compound OLS estimators (yellow box-plots),
OLS estimators (red box-plots), {\color{red}S}%
BLUE estimators (blue box-plots).
Top-left panel: scenario 1, independence. Top-right panel: scenario 2, mild dependence. 
Bottom panel: scenario 3, full dependence.}\label{fig:MC1}
\end{figure}

From the comparison of the box-plots corresponding to $\lambda=0$ and $\lambda=1$ with the other cases, we can observe that it is always more convenient to use the whole information on $y^{(f)}(t)$ and $y^{(d)}(t)$ (this behaviour is more evident in scenario 1). Among the three estimators $\hat\bbeta^{(c)}_{\lambda}(t) $, $\hat{\bbeta}_\lambda(t) $ and 
${\HbbetaR}(t)$, the 
{\color{red}S}BLUE is the most precise, as expected. 
When there is a one-to-one dependence between  $y^{(f)}(t)$ and $y^{(d)}(t)$, one source of information is redundant and all the functional estimators coincide (bottom panel of Figure~\ref{fig:MC1}). 

\subsection{Spread of functions and derivatives}
Also in this case, we consider three different scenarios. Let 
$$
r_{ll} =  \frac{\big(\Sigma_{ \hat\bbeta^{(f)} }\big)_{ll}}{\big(\Sigma_{ \hat\bbeta^{(d)} }\big)_{ll}}, \qquad l=1,\ldots,p,
$$ 
where $\Sigma_{ \cdot }$ denotes the  the $H^1$-global covariance matrix  defined in \eqref{eq:sigma}.
We generate functional data $y_{ij}^{(f)}(t)$ and $y_{ij}^{(d)}(t)$ with a different spread, such that
\begin{enumerate}
\item
$r_{ll} \cong 0.25 $ (in this sense, $y_{ij}^{(f)}(t)$ is ``more concentrate'' than $y_{ij}^{(d)}(t)$);
\item
$r_{ll} \cong 1 $ ($y_{ij}^{(f)}(t)$ and $y_{ij}^{(d)}(t)$ have more or less the same spread);
\item
$r_{ll} \cong 4 $ ($y_{ij}^{(d)}(t)$ is ``more concentrate'' than $y_{ij}^{(f)}(t)$).
\end{enumerate}  

As before, the performance of the different estimators is evaluated by comparing the $H^1$-norm of the $p$-components of the
estimation errors. Figures~\ref{fig:MC2} depicts the Montecarlo distribution of the $H^1$-norm of the first component: 
$\| \hat{\bbeta}_{\lambda,1}(t) - {\bbeta}_1(t) \|_{H^1} $ for different values of 
$\lambda$ (red box-plot,  \eqref{eq:Hbbeta}), $\| \hat\bbeta^{(c)}_{\lambda,1}(t) - {\bbeta}_1(t) \|_{H^1} $ 
for different values of 
$\lambda$ (yellow box-plots) and 
$\| {\HbbetaR}_1(t) - {\bbeta}_1(t) \|_{H^1} $ (blue box-plot).
\begin{figure}
\begin{center}
\fbox{\includegraphics[width=.48\textwidth]{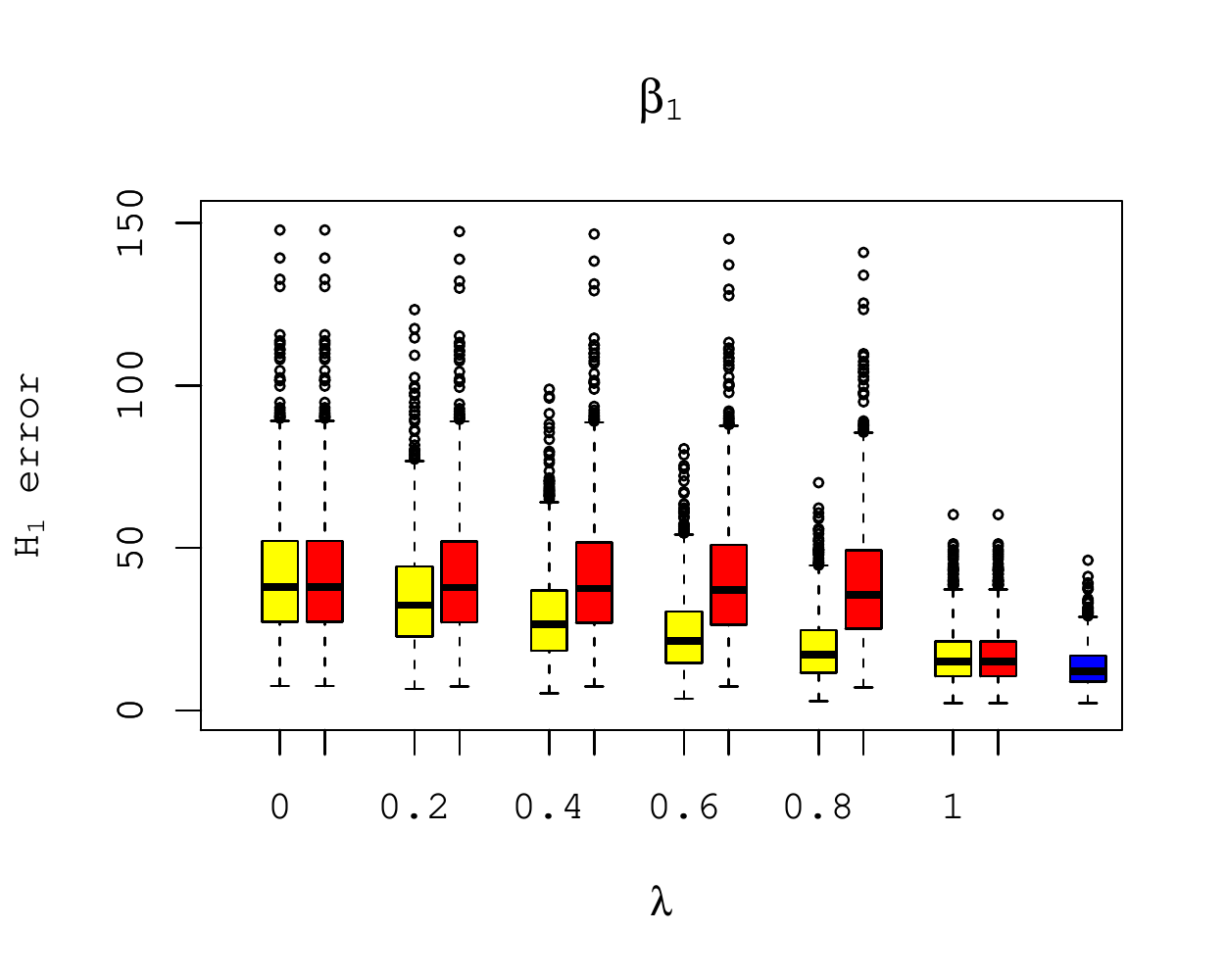}} 
\fbox{\includegraphics[width=.48\textwidth]{fig2B.pdf}}
\fbox{\includegraphics[width=.48\textwidth]{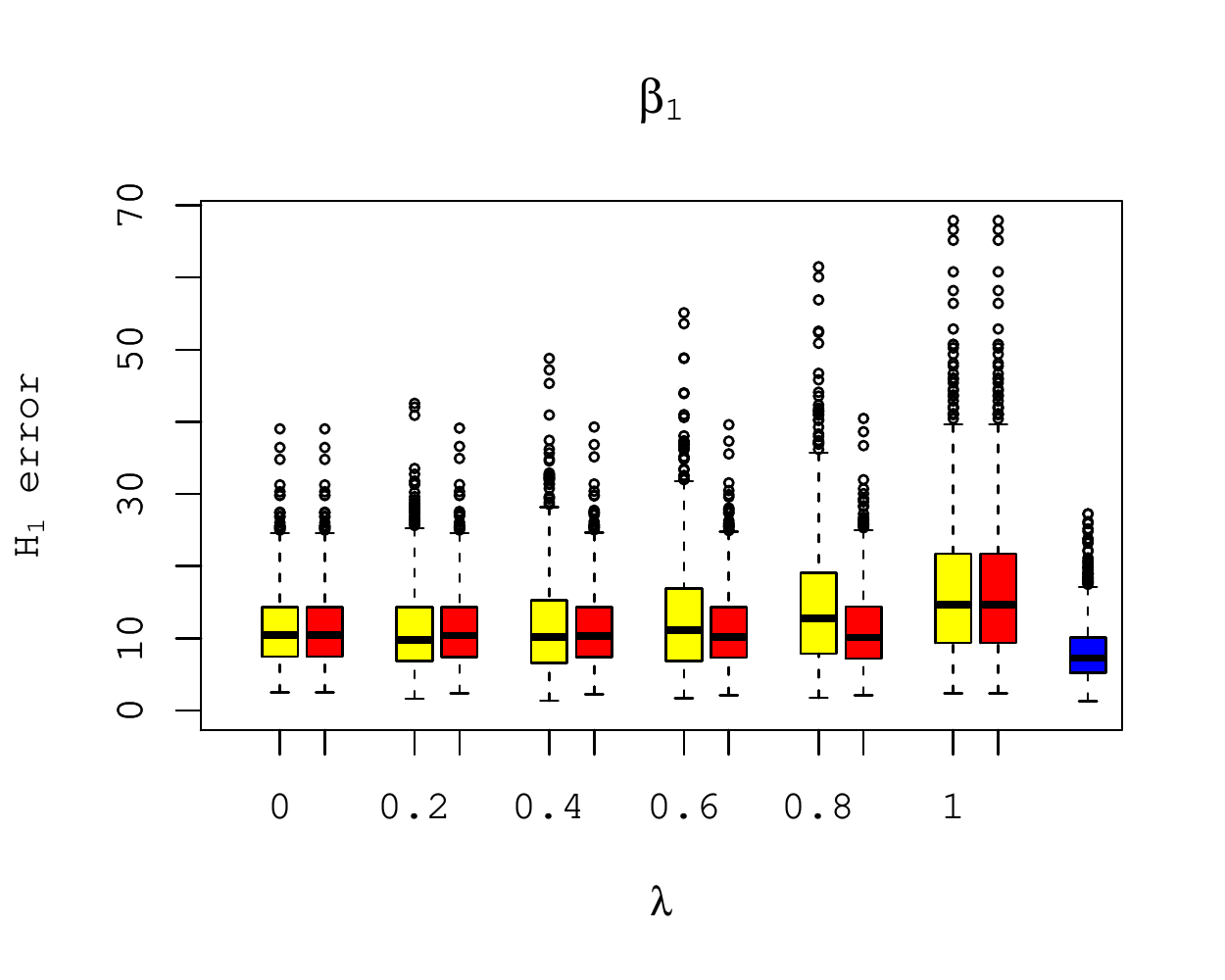}}
\end{center}
\caption{$H^1$ norm of the first components estimation errors, for compound OLS estimators (yellow box-plots),
OLS estimators (red box-plots), {\color{red}S}%
BLUE estimators (blue box-plots).
Top-left panel: scenario 1, independence. Top-right panel: scenario 2, mild dependence. 
Bottom panel: scenario 3, full dependence.}\label{fig:MC2}
\end{figure}

From the comparison of the box-plots of  $\hat\bbeta^{(c)}_{\lambda}(t) $ and $\hat{\bbeta}_\lambda(t) $ corresponding to $\lambda=0$ and $\lambda=1$ with the other cases, 
it seems more convenient to use just the less ``less spread'' information: $y^{(f)}(t)$ in Scenario 1 and $y^{(d)}(t)$ in Scenario 2. 
Comparing the precision of $\hat\bbeta^{(c)}_{\lambda}(t) $ and $\hat{\bbeta}_\lambda(t) $ with the one of the ${\HbbetaR}(t)$, however, the 
{\color{red}S}BLUE is the most precise, as expected. 
Hence, we suggest the use of the whole available information through the use of the {\color{red}S}BLUE.
Of course, when one of the sources of information has a spread near to zero then the more precise estimator is the one that uses just that piece of information and 
${\HbbetaR}(t)$ reflects this behaviour.

\section{Summary}\label{sect:discuss}
Functional data are suitably modelled in separable Hilbert spaces (see \citet{HorKok:HK12} and \citet{Bosq00}) and  $L^2$ is usually sufficient to handle the majority of the techniques proposed in the literature of functional data analysis.

Differently, we consider proper Sobolev spaces, since we guess that the data may provide information on both curve functions and their derivatives. The classical theory for linear regression models is extended to this context by means of the sample of Riesz representatives. Roughly speaking, the Riesz representatives are ``quantities'' which incorporate both functions and derivatives information in a non trivial way. \rosso{More in detail, a generalization of the Riesz representatives is proposed to take into account the possible correlation between curves and derivatives. These generalized Riesz representatives are called just \lq\lq representatives''.
Using a sample of representatives, we prove a strong, generalized version of the well known Gauss-Markov theorem for functional linear regression models.
 Despite the complexity of the problem we obtain an elegant and simple solution, through the use of the representatives which belong to a Sobolev space. This result states that the proposed estimator, which takes into account both information about curves and derivatives (throughout the representatives), is much \rosso{more} efficient than the usual OLS estimator based only on one sample of functions (curves or derivatives). 
 The superiority of the proposed estimator is also showed in the simulation study described in Section \ref{sect:simul}.}

\appendix

\section{Proofs}

\begin{proof}[Proof of Theorem~\ref{thm:Riesz_representation}]
Part a). We consider the sum of square residuals:
\begin{multline*}
\begin{aligned}
S\Big( \bbeta(t) \Big) &= \sum_{j=1}^r \sum_{i=1}^n \Big( \|
{y}_{ij}^{(f)}(t) - \mathbf{f}(\mathbf{x}_i)^T {\bbeta}(t)
\|^2_{L^2} + \|
{y}_{ij}^{(d)}(t) - \mathbf{f}(\mathbf{x}_i)^T {\bbeta}'(t)
\|^2_{L^2}
\Big) \\
& = \sum_{j=1}^r \sum_{i=1}^n \Big( \langle
{y}_{ij}^{(f)}(t) - \mathbf{f}(\mathbf{x}_i)^T {\bbeta}(t) ,
{y}_{ij}^{(f)}(t) - \mathbf{f}(\mathbf{x}_i)^T {\bbeta}(t)
\rangle_{L^2} \\
& \qquad\qquad+ \langle
{y}_{ij}^{(d)}(t) - \mathbf{f}(\mathbf{x}_i)^T {\bbeta}'(t) ,
{y}_{ij}^{(d)}(t) - \mathbf{f}(\mathbf{x}_i)^T {\bbeta}'(t)
\rangle_{L^2}
\Big)
\end{aligned}
\end{multline*}

The G\^ateaux derivative
of $S( \cdot )$ at $\bbeta(t)$ in the direction of $\mathbf{g}(t)\in (H^1)^p$ is
\begin{align}
\notag
\lim_{h\to 0} \frac{S( \bbeta(t) + h \mathbf{g}(t)) - S( \bbeta(t) ) }{h}  = &
2 \bigg(
\sum_{j=1}^r \sum_{i=1}^n \Big( \langle
{y}_{ij}^{(f)}(t) - \mathbf{f}(\mathbf{x}_i)^T {\bbeta}(t) , \mathbf{f}(\mathbf{x}_i)^T {\mathbf{g}}(t)
\rangle_{L^2}
\\ \notag
&   +\langle
{y}_{ij}^{(d)}(t) - \mathbf{f}(\mathbf{x}_i)^T {\bbeta}'(t) , \mathbf{f}(\mathbf{x}_i)^T {\mathbf{g}'}(t)
\rangle_{L^2}
\Big)
\bigg)
\\
= &
\label{eq:derivata}
2r
\Big( \langle
F^T \bar{\mathbf{y}}^{(f)}(t)- F^TF {\bbeta}(t) , {\mathbf{g}}(t)
\rangle_{(L^2)^p}
\\ \notag
&
+\langle
F^T \bar{\mathbf{y}}^{(d)}(t) - F^TF  {\bbeta}'(t) , {\mathbf{g}'}(t)
\rangle_{(L^2)^p}
\Big),
\end{align}
where $\bar{\mathbf{y}}^{(f)}(t)$ and $\bar{\mathbf{y}}^{(d)}(t)$ are two $n\times 1$ vectors whose $i$-th elements are
\begin{equation}\label{eq:meanObs}
\bar{\mathbf{y}}^{(f)}_i(t) = \frac{\sum_{j=1}^r {y}_{ij}^{(f)}(t)}{r} , \qquad
\bar{\mathbf{y}}^{(d)}_i(t) = \frac{\sum_{j=1}^r {y}_{ij}^{(d)}(t)}{r}.
\end{equation}
Developing the right-hand side of (\ref{eq:derivata}), we have that the G\^ateaux derivative is
\begin{align}
 \notag
= &
2r \bigg(
\Big( \langle
F^T  \bar{\mathbf{y}}^{(f)}(t)  , {\mathbf{g}}(t)
\rangle_{(L^2)^p}  +\langle
F^T \bar{\mathbf{y}}^{(d)}(t)  , {\mathbf{g}'}(t)
\rangle_{(L^2)^p}
\Big)
\\ \notag
&
-\Big( \langle
F^TF {\bbeta}(t) , {\mathbf{g}}(t)
\rangle_{(L^2)^p}  +\langle
F^TF  {\bbeta}'(t) , {\mathbf{g}'}(t)
\rangle_{(L^2)^p}
\Big)
\bigg)
\\ \label{eq:GD}
= &
2r
\Big( \langle
F^T  \bar{\mathbf{y}}(t)  , {\mathbf{g}}(t)
\rangle_{(H^1)^p}
- \langle
F^TF {\bbeta}(t) , {\mathbf{g}}(t)
\rangle_{(H^1)^p}
\Big),
\end{align}
where $\bar{\mathbf{y}}(t)$ is a $n\times 1$ vector whose $i$-th element is the Riesz representative of $\Big(\bar{\mathbf{y}}^{(f)}_i(t),\bar{\mathbf{y}}^{(d)}_i(t) \Big)$.

The G\^ateaux derivative \eqref{eq:GD} is equal to $0$ for any ${\mathbf{g}}(t)\in (H^1)^p$ if and only if  ${\Hbbeta} (t)$
is given by the following equation:
\[
F^TF{\Hbbeta} (t) = F^T \bar{\mathbf{y}}(t),
\]
which proves the first statement of the theorem.

\bigskip

Part b) Definition \ref{DEf:def3} and  model \eqref{eq:rMeas} imply that, for any $h(t)\in H^1$,
\begin{align*}
\big\langle E(\tilde{y}_{ij}(t)),h(t) \big\rangle_{H^1} & =
E\big(\langle {y}^{(f)}_{ij}(t),h(t) \rangle_{L^2}\big) + E\big(\langle {y}^{(d)}_{ij}(t),h'(t) \rangle_{L^2}\big) \\
&=
\big\langle \mathbf{f}(\mathbf{x}_i)^T {\bbeta}(t) , h(t)\big\rangle_{H^1},
\end{align*}
then $ E(\bar{\mathbf{y}}(t)) = F {\bbeta}(t) $, and hence $\Hbbeta(t)$ is unbiased.
Moreover,
\begin{equation}\label{eq:Errors_H1}
\bar{{y}}_i(t) - \mathbf{f}(\mathbf{x}_i)^T {\bbeta}(t)  = \tilde{\alpha}_i(t) + \frac{\sum_{j=1}^r \tilde{\varepsilon}_{ij}(t)}{r},
\qquad i = 1,\ldots, n
\end{equation}
where $\tilde{\alpha}_i(t) $ and $\tilde{\varepsilon}_{ij}(t)$ denote the Riesz representatives of $({\alpha}^{(f)}_i(t),{\alpha}^{(d)}_i(t)) $ and $({\varepsilon}_{ij}^{(f)}(t),{\varepsilon}_{ij}^{(d)}(t))$, respectively. From the hypothesis (\ref{hp:1}) and (\ref{hp:2}) in the model \eqref{eq:rMeas}, the left-hand side quantities in \eqref{eq:Errors_H1} are zero-mean i.i.d.\ processes, for $ i = 1,\ldots, n $.
Therefore, the \rosso{global} covariance matrix of $\bar{\mathbf{y}}(t)$ is $\sigma^2 I_n$, where
$\sigma^2 = E(\| \bar{{y}}_i(t) - \mathbf{f}(\mathbf{x}_i)^T {\bbeta}(t) \|^2_{H^1})$. Hence, the \rosso{global} covariance matrix of $\Hbbeta (t)$ is
$\Sigma_{\Hbbeta}= \sigma^2 (F^TF)^{-1}$.
\end{proof}

\begin{proof}[Proof of Lemma~\ref{lem:eRinH1}] We have that
\begin{align*}
E\| e_i^{(\mathbf{R})} \|^2_{{H}^1} & = E \sum_{h}
\Big\langle \boldsymbol\Psi_{h}, \sum_k \sqrt{\lambda_k} e_{i,k} ({\psi^{(\mathbf{R})}_k},{\psi^{(\mathbf{R})}_k}') \Big\rangle_{\mathbf{L}^2}^2\\
& = \sum_{h} E  \sum_{k_1,k_2} \sqrt{\lambda_{k_1}}\sqrt{\lambda_{k_2}}  e_{i,k_1} e_{i,k_2} 
\langle \boldsymbol\Psi_{h}, ({\psi^{(\mathbf{R})}_{k_1}},{\psi^{(\mathbf{R})}_{k_1}}') \rangle_{\mathbf{L}^2}
\langle \boldsymbol\Psi_{h}, ({\psi^{(\mathbf{R})}_{k_2}},{\psi^{(\mathbf{R})}_{k_2}}') \rangle_{\mathbf{L}^2}\\
& = \sum_{k,h}
\lambda_k \langle \boldsymbol\Psi_{h}, ({\psi^{(\mathbf{R})}_k},{\psi^{(\mathbf{R})}_k}') \rangle_{\mathbf{L}^2}^2 
= \sum_{k,h}
\lambda_h ( \langle \boldsymbol\Phi_{h}, ({\phi^{(\mathbf{R})}_k},{\phi^{(\mathbf{R})}_k}') \rangle_{\mathbf{L}^2_{\mathbf{R}}}  )^2  
\intertext{From \eqref{eq:tricks}, the last term is equal to $\sum_{k,h}
\lambda_h ( \langle ({\phi^{(\mathbf{R})}_h},{\phi^{(\mathbf{R})}_h}') , \boldsymbol\Phi_{k}\rangle_{\mathbf{L}^2_{\mathbf{R}}}  )^2  $.
Hence,}
E\| e_i^{(\mathbf{R})} \|^2_{{H}^1}  & = \sum_{k,h}
\lambda_h ( \langle ({\phi^{(\mathbf{R})}_h},{\phi^{(\mathbf{R})}_h}') , \boldsymbol\Phi_{k}\rangle_{\mathbf{L}^2_{\mathbf{R}}}  )^2  
 = \sum_{h}
\lambda_h \sum_{k} ( \langle ({\phi^{(\mathbf{R})}_h},{\phi^{(\mathbf{R})}_h}'), \boldsymbol\Phi_{k}\rangle_{\mathbf{L}^2_{\mathbf{R}}}  )^2  \\
& = \sum_{h}
\lambda_h \| ({\phi^{(\mathbf{R})}_h},{\phi^{(\mathbf{R})}_h}') \|_{\mathbf{L}^2_{\mathbf{R}}}^2  \leq \sum_{h}
\lambda_h ,
\end{align*}
where the last inequality follows from \eqref{eq:projH1Rbase}. Since $\sum_{h}\lambda_h < \infty$, we get the thesis.
\end{proof}

\subsection{Proof of Theorem~\ref{GM-Repeated}}
\label{appendix}
The estimator $\HbbetaR(t)$ is a linear map which associates an element $\HbbetaR(t)$ in $(H^1)^{p}$
to any $n r$-tuple $(y_{ij}^{(f)}(t), y_{ij}^{(d)}(t))$. 
In what follows, we show that it is the ``best'' among all the linear unbiased closed operators
$\mathbf{C}: {\mathcal{R}}\to (H^1)^p$.

The model \eqref{eq:rMeas} may be written in the following vectorial form:
\begin{equation}\label{eq:rMeasVect}
\left\{
\begin{aligned}
{\mathbf{y}}^{(f)} (t) & = (F \otimes \mathbf{1}_r) {\bbeta}(t) + ({\boldsymbol{\alpha}}^{(f)}(t) \otimes \mathbf{1}_r) +
 {\boldsymbol{\varepsilon}}^{(f)}(t)\\
{\mathbf{y}}^{(d)} (t) & = (F \otimes \mathbf{1}_r) {\bbeta}'(t) + ({\boldsymbol{\alpha}}^{(d)}(t) \otimes \mathbf{1}_r) +
 {\boldsymbol{\varepsilon}}^{(d)}(t)\\
\end{aligned}
\right.
\end{equation}
where $\mathbf{1}_r $ is the column vector of  length $r$ with all components equal to $1$.\\
\rosso{In general, if 
	 $$\mathbf{y}^{(1)}(t) = \Big({y}^{(1)}_{11} (t), \ldots ,{y}^{(1)}_{1r} (t),
{y}^{(1)}_{21} (t), \ldots ,{y}^{(1)}_{2r} (t),\ldots,{y}^{(1)}_{n1} (t), \ldots ,{y}^{(1)}_{nr} (t)\Big)^T$$
and
$$\mathbf{y}^{(2)}(t) = \Big({y}^{(2)}_{11} (t), \ldots ,{y}^{(2)}_{1r} (t),
{y}^{(2)}_{21} (t), \ldots ,{y}^{(2)}_{2r} (t),\ldots,{y}^{(2)}_{n1} (t), \ldots ,{y}^{(2)}_{nr} (t)\Big)^T$$
are two} $nr\times 1$ block vectors in $\mathcal{R}$, we may define the following $n$ dimensional vector
\begin{equation}
\bar{\mathbf{y}}^{\rosso{(1,2)(\mathbf{R})}}(t) = \Big(\bar{y}^{\rosso{(1,2)(\mathbf{R})}}_1 (t), 
\ldots ,\bar{y}^{\rosso{(1,2)(\mathbf{R})}}_n (t)\Big)^T,
\label{eq:riesz12}
\end{equation}
where $\bar{y}^{\rosso{(1,2)(\mathbf{R})}}_i (t) $ is the $H^1_{\mathbf{R}}$ representation of
\[
\Big(\frac{\sum_{j=1}^r {y}^{(1)}_{ij} (t)  }{r} ,
\frac{\sum_{j=1}^r {y}^{(2)}_{ij} (t)  }{r}
\Big) .
\]
Now we can introduce the following linear operator
\begin{equation} \label{eq:def_Dop}
\mathbf{D}\Big({\mathbf{y}}^{(1)} (t),{\mathbf{y}}^{(2)} (t)\Big) = \mathbf{C}\Big({\mathbf{y}}^{(1)} (t),{\mathbf{y}}^{(2)} (t) \Big)
-(F^T F)^{-1} F^T \ \bar{\mathbf{y}}^{\rosso{(1,2)(\mathbf{R})}} (t).
\end{equation}
Hence,
\begin{align} \label{eq:def_D}
\mathbf{D}({\mathbf{y}}^{(f)} (t),{\mathbf{y}}^{(d)} (t)) &= \mathbf{C}({\mathbf{y}}^{(f)} (t),{\mathbf{y}}^{(d)} (t))
-(F^T F)^{-1} F^T \bar{\mathbf{y}}^{(\mathbf{R})}(t) \\
& = \mathbf{C}({\mathbf{y}}^{(f)} (t),{\mathbf{y}}^{(d)} (t)) - \HbbetaR (t) \notag
\end{align}
and
$$
\mathbf{C}({\mathbf{y}}^{(f)} (t),{\mathbf{y}}^{(d)} (t)) = \mathbf{D}({\mathbf{y}}^{(f)} (t),{\mathbf{y}}^{(d)} (t))+\HbbetaR(t).
$$
The thesis follows immediately if we prove that ${\mathrm{O}}(\mathbf{D}({\mathbf{y}}^{(f)} (t),{\mathbf{y}}^{(d)} (t)))$ and
${\mathrm{O}}(\HbbetaR(t))$ are uncorrelated.

Since both $\mathbf{B}(t)$ and $\HbbetaR (t)$ are unbiased, $E\Big(\mathbf{D}({\mathbf{y}}^{(f)} (t),{\mathbf{y}}^{(d)} (t))\Big)= \mathbf{0}$, and thus
we have to prove that
\begin{equation}\label{eq:scoll}
E\big\langle\ {\mathrm{O}} (\mathbf{D} ({\mathbf{y}}^{(f)} (t),{\mathbf{y}}^{(d)} (t)) )  \ ,\
{\mathrm{O}} ({\HbbetaR}(t) - {\bbeta} (t) ) \ \big\rangle_{H^1} =0,
\end{equation}
for any choice of linear operator ${{\mathrm{O}}}:(H^1)^p \to H^1$.

The proof of equality (\ref{eq:scoll}) is developed in four steps.

\bigskip

\noindent \textbf{First step}.
\emph{The goal of this step is to prove that $\mathbf D$ applied to the deterministic part of the model
$\Big((F \otimes \mathbf{1}_r) {\bbeta}(t), (F \otimes \mathbf{1}_r) {\bbeta}'(t)\Big)$ is identically null.
As a consequence,
\begin{equation}\label{eq:r0matrix001}
\mathbf D \Big({\mathbf{y}}^{(f)} (t),{\mathbf{y}}^{(d)} (t)\Big) =
\mathbf D \Big({\boldsymbol{\alpha}}^{(f)} (t) \otimes \mathbf{1}_r + {\boldsymbol{\varepsilon}}^{(f)} (t) ,
{\boldsymbol{\alpha}}^{(d)} (t) \otimes \mathbf{1}_r + {\boldsymbol{\varepsilon}}^{(d)} (t)\Big)  .
\end{equation}}

\medskip

From the linearity of the closed operator $\mathbf{C}$, and the
zero-mean hypothesis (\ref{hp:1}) and (\ref{hp:2}), we have that
\begin{align}
\notag E \Big( \mathbf{C}\big({\mathbf{y}}^{(f)} (t),{\mathbf{y}}^{(d)} (t) \big) \Big) & =
E \Big( \mathbf{C}\big( (F \otimes \mathbf{1}_r) {\bbeta}(t) + ({\boldsymbol{\alpha}}^{(f)}(t) \otimes \mathbf{1}_r) +
 {\boldsymbol{\varepsilon}}^{(f)}(t), \\
\notag & \qquad\qquad  (F \otimes \mathbf{1}_r) {\bbeta}'(t) + ({\boldsymbol{\alpha}}^{(d)}(t) \otimes \mathbf{1}_r) +
 {\boldsymbol{\varepsilon}}^{(d)}(t) \big) \Big)\\
\notag & = \mathbf{C}((F \otimes \mathbf{1}_r) {\bbeta}(t), (F \otimes \mathbf{1}_r) {\bbeta}'(t) ).
\end{align}
Since $E \Big( \mathbf{C}\big({\mathbf{y}}^{(f)} (t),{\mathbf{y}}^{(d)} (t) \big) \Big)={\bbeta}(t)$
we have that
\begin{equation}
\label{eq:DF0_1}
\mathbf{C}\Big((F \otimes \mathbf{1}_r) {\bbeta}(t), (F \otimes \mathbf{1}_r) {\bbeta}'(t) \Big) = {\bbeta}(t)
\end{equation}
In addition, from the definition (\ref{eq:riesz12}) if
$$
{\mathbf{y}}^{(1)} (t) =  F{\bbeta}(t) \otimes \mathbf{1}_r\quad {\rm and} \quad
\ {\mathbf{y}}^{(2)} (t)=  F{\bbeta}'(t) \otimes \mathbf{1}_r
$$
then
\begin{equation}\label{eq:DF0_2}
\bar{\mathbf{y}}^{\rosso{\rosso{(1,2)(\mathbf{R})}}} (t) =  F{\bbeta}(t).
\end{equation}
Combining \eqref{eq:def_Dop}, \eqref{eq:DF0_1} and \eqref{eq:DF0_2} gives
\begin{equation}\label{eq:DF0}
\mathbf{D}((F \otimes \mathbf{1}_r) {\bbeta}(t), (F \otimes \mathbf{1}_r) {\bbeta}'(t) )=\mathbf{0},
\end{equation}
and hence \eqref{eq:r0matrix001}.

\bigskip

\noindent \textbf{Second step}. \emph{Representation of the linear operator $D_{l}$.}

\medskip

For the linearity of the $l$-th component of $\mathbf{D}$ with respect to the bivariate
observations $\Big({y}^{(1)}_{ij}(t),{y}^{(2)}_{ij}(t)\Big)$:
\begin{equation}\label{eq:rep_D1}
D_{l}\Big({\mathbf{y}}^{(1)} (t),{\mathbf{y}}^{(2)} (t)\Big) =
 \sum_{i,j} D_{l,ij}\Big({y}^{(1)}_{ij}(t),{y}^{(2)}_{ij}(t)\Big),
\end{equation}
where, for any $i=1,\ldots,n $ and $j=1,\ldots, r$, $D_{l,ij}$ is linear. The domain of
$D_{l,ij}$ is contained in $L^2(\mathbf{R}^2)$.
Let \rosso{$(\phi_{g} )_g$ be
an orthonormal base of $H^1_{\mathbf{R}}$}.
\rosso{We express the linear operator $y = D_{l,ij}(x)$ in terms of the base $(\boldsymbol\Psi_{k} )_k$ 
for $x$ and $(\phi_{g} )_g$ for $y$. In fact, $\mathcal{R}\subseteq (\mathbf{L}^2)^{n r}$ and
$y\in H^1\subseteq K^*$ (see \eqref{eq:H1RRreprK}). Accordingly,
\begin{equation}\label{eq:rep_D2}
D_{l,ij}({y}^{(1)}_{ij}(t),{y}^{(2)}_{ij}(t)) \\
= \sum_{k,g}
\langle \boldsymbol\Psi_{k} , ({y}^{(1)}_{ij}(t) \,,\, {y}^{(2)}_{ij}(t))^T \rangle_{\mathbf{L}^2} 
\ d^{k,g}_{l,ij} \ \phi_{g}(t),
\end{equation}
where
\[
d^{k,g}_{l,ij}  = \langle D_{l,ij}( {\boldsymbol\Psi}_{k} )(t) ,  \phi_{{g}}(t) \rangle_{H^1_{\mathbf{R}}}.
\]}

\bigskip

\noindent\textbf{Third step.} \emph{Proof of
\rosso{%
\begin{equation*}
\sum_k \sum_{i=1}^n \sum_{j=1}^r {m}_{{l_2},i} \, d^{k,g}_{l_1,ij}  
\langle \boldsymbol\Psi_{k} \,,\,  (h,h')^T \rangle_{\mathbf{L}^2}
= 0, \qquad   g,l_1,l_2,\ h\in H^1,
\end{equation*}
where $\mathbf{m}_{l_2}= ({m}_{{l_2},1}, \ldots , {m}_{{l_2},n})^T$ is the ${l_2}$-th row of $(F^TF)^{-1}F^T$.
In particular, since $H^1_{\mathbf{R}} \subseteq H^1$,
\begin{equation}\label{eq:r0matrix1}
\sum_{i,j,k}^n {m}_{{l_2},i} \, d^{k,g}_{l_1,ij}  
\langle \boldsymbol\Psi_{k} \,,\,  {m}_{{l_2},i} (h,h')^T \rangle_{\mathbf{L}^2}
= 0, \qquad   g,l_1,l_2,\ h\in H^1_{\mathbf{R}}.
\end{equation}
}}
\medskip

Let $\mathbf{h}^{({l_2})}(t) \in (H^1)^p$ be the null vector
except for the ${l_2}$-th component which is $h(t)\in H^1$, and let
$\mathbf{h}(t)= (F^TF)^{-1} \mathbf{h}^{({l_2})}(t) \in (H^1)^p$.
Setting ${\bbeta}(t)=\mathbf{h}(t)$ in \eqref{eq:DF0},
\begin{align}
0 & = D_{l_1} ((F \otimes \mathbf{1}_r) \mathbf{h}(t), (F \otimes \mathbf{1}_r) \mathbf{h}' (t) )
\nonumber
\\
& = D_{l_1} ((F \mathbf{h} (t)) \otimes \mathbf{1}_r , (F \mathbf{h}' (t)) \otimes \mathbf{1}_r  )\nonumber
\\
& = D_{l_1} (F (F^TF)^{-1} \mathbf{h}^{({l_2})} (t) \otimes \mathbf{1}_r , F (F^TF)^{-1} {\mathbf{h}^{({l_2})}}' (t) \otimes \mathbf{1}_r )\nonumber
\\
& = D_{l_1} ( h(t) \mathbf{m}_{l_2}  \otimes \mathbf{1}_r , h'(t) \mathbf{m}_{l_2}  \otimes \mathbf{1}_r )\nonumber
\\
&= \sum_{i=1}^n \sum_{j=1}^r D_{l_1,ij} ( h(t) {m}_{{l_2},i} \,,\,  h'(t) {m}_{{l_2},i} ) \nonumber \\
& \rosso{%
= \sum_{g} \Big(\sum_{k,i,j}
 (\langle \boldsymbol\Psi_{k} , ({m}_{{l_2},i} h \,,\, {m}_{{l_2},i} h')^T \rangle_{\mathbf{L}^2})
\, d^{k,g}_{l_1,ij} \Big) \phi_{{g}}(t),}
\label{eq:equality1}
\end{align}
where the last equality is due to (\ref{eq:rep_D2}).

\bigskip

\noindent\textbf{Fourth step.} \emph{Proof of \eqref{eq:scoll}:
\begin{equation*}
E\big\langle\ {\mathrm{O}} (\mathbf{D} ({\mathbf{y}}^{(f)} (t),{\mathbf{y}}^{(d)} (t)) )  \ ,\
{\mathrm{O}} ({\HbbetaR}(t) - {\bbeta} (t) ) \ \big\rangle_{H^1} =0,
\end{equation*}
for any choice of linear operator ${{\mathrm{O}}}:(H^1)^p \to H^1$.}

\medskip

From Theorem~\ref{cor:yR-eR} and  from \eqref{eq:HbbetaR}, 
$\HbbetaR (t) - \bbeta (t) =  (F^TF)^{-1}F^T {\mathbf{e}^{\mathbf(R)}}$, and hence
\begin{multline}
E\big\langle\ {\mathrm{O}} (\mathbf{D} ({\mathbf{y}}^{(f)} (t),{\mathbf{y}}^{(d)} (t)) )  \ ,\
{\mathrm{O}} ({\HbbetaR}(t) - {\bbeta} (t) ) \ \big\rangle_{H^1}  \\
\begin{aligned}
& =
E\big\langle\ {\mathrm{O}} (\mathbf{D} ({\mathbf{y}}^{(f)} (t),{\mathbf{y}}^{(d)} (t)) )  \ ,\
{\mathrm{O}} ( (F^TF)^{-1} F^T {\mathbf{e}^{\mathbf(R)}}(t) ) \ \big\rangle_{H^1}
\\
& =
\label{Eq:equation2}
E\Big\langle {\mathrm{O}}( \mathbf D ({\boldsymbol{\alpha}}^{(f)} (t) \otimes \mathbf{1}_r + {\boldsymbol{\varepsilon}}^{(f)} (t) ,
{\boldsymbol{\alpha}}^{(d)} (t) \otimes \mathbf{1}_r + {\boldsymbol{\varepsilon}}^{(d)} (t) ) )  \ , \\
& \qquad\qquad
{\mathrm{O}} ( (F^TF)^{-1} F^T {\mathbf{e}^{\mathbf(R)}}(t) ) \Big\rangle_{H^1},
\end{aligned}
\end{multline}
where the last equality is a consequence of \eqref{eq:r0matrix001}.

\rosso{Since $x\in (H^1)^p\subseteq (K^*)^p$ (see \eqref{eq:H1RRreprK}),
we express the linear operator $y = {\mathrm{O}} (x)$ in terms of the base 
$(\phi_{g_1} \times \phi_{g_2} \times \cdots \times \phi_{g_p} )_{g_1,\ldots,g_p}$ for $x$ 
and $(\zeta_{h} )_h$ for $y$, where  $(\zeta_h)_h$ is an orthonormal base of $H^1$.}
To begin with, from the linearity of the operator ${{\mathrm{O}}}:(H^1)^p \to H^1$, we have that
\[
{\mathrm{O}}( b_1(t), \ldots , b_p(t))
= \sum_{l=1}^p
\mathrm{O}( \underbrace{0, \ldots , 0}_{l-1\text{ times}}, b_l(t), \underbrace{0, \ldots , 0}_{p-l\text{ times}}) .
\]
Since $  b_l(t) = \sum_g \langle b_l(t)  , \phi_{g} (t) \rangle_{H^1_{\mathbf{R}}} \;\phi_{g} (t) = \sum_g b_l^g \phi_{g} (t) $,
where $ b_l^g = \langle b_l(t)  , \phi_{g} (t) \rangle_{H^1_{\mathbf{R}}} $,
we have
\[
{\mathrm{O}}( b_1(t), \ldots , b_p(t))
= \sum_{l,g} b_l^g \;
\mathrm{O}( \underbrace{0, \ldots , 0}_{l-1\text{ times}}, \phi_g(t), \underbrace{0, \ldots , 0}_{p-l\text{ times}}) .
\]
Now, setting
\[
O^{g,h}_{l} = \big\langle\mathrm{O}( \underbrace{0, \ldots , 0}_{l-1\text{ times}}, \phi_g(t), \underbrace{0, \ldots , 0}_{p-l\text{ times}}) , \zeta_{h}(t)\big\rangle_{H^1} ,
\]
then we have the representation of $\mathrm{O}$ in terms of the required bases:
\[
{\mathrm{O}}( b_1(t), \ldots , b_p(t))
= \sum_{l,g,h}
b_l^g \;O^{g,h}_{l}\;  \zeta_{h}(t)  .
\]
Hence, from Equations (\ref{Eq:equation2}),  \eqref{eq:rep_D1} and \eqref{eq:rep_D2}, the thesis \eqref{eq:scoll}
becomes
\begin{multline*}
E\Big\langle\ \sum_{l,g,h}
\Big(
\sum_{i,j,k} (\langle \boldsymbol \Psi_{k} \,,\,   ({\alpha}_{i}^{(f)}(t) + {\varepsilon}_{ij}^{(f)}(t) ,
{\alpha}_{i}^{(d)}(t) + {\varepsilon}_{ij}^{(d)}(t) )^T \rangle_{\mathbf{L}^2})\,
d^{k,g}_{l_1,ij} \Big)  O^{g,h}_{l} \;\zeta_{h}(t) \ ,
\\
\sum_{l,g,h}
\Big( \big\langle{\mathbf{e}^{\mathbf(R)}}(t)^T \mathbf{m}_{l}  \,,\,  \phi_{g} (t)\big\rangle_{H^1_{\mathbf{R}}} \Big)
O^{g,h}_{l} \;\zeta_{h}(t)
\ \Big\rangle_{H^1} =0 ,
\end{multline*}
From \eqref{eq:rGMprod} and \eqref{eq:innnerOmega}, since $\langle \zeta_{h_1},\zeta_{h_2}\rangle_{H^1} = \delta_{h_1}^{h_2},$
the left-hand side of the last equation becomes
\begin{align*}
&E\Big\langle\ \sum_{l,g,h}
O^{g,h}_{l} \;\zeta_{h}(t)
\\
&  \qquad
\Big(
\sum_{i,j,k} (\langle \boldsymbol \Psi_{k} \,,\,   ({\alpha}_{i}^{(f)}(t) + {\varepsilon}_{ij}^{(f)}(t) ,
{\alpha}_{i}^{(d)}(t) + {\varepsilon}_{ij}^{(d)}(t) )^T \rangle_{\mathbf{L}^2})\,
d^{k,g}_{l,ij} \Big) \ ,
\\
& \qquad \qquad
\sum_{l,g,h}
O^{g,h}_{l} \;\zeta_{h}(t)
\big\langle{\mathbf{e}^{\mathbf(R)}}(t)^T \mathbf{m}_{l_2}  \,,\,  \phi_{g} (t)\big\rangle_{H^1_{\mathbf{R}}}
\ \Big\rangle_{H^1} \\
&  =
E\Big( \sum_{l_1,l_2,g_1,g_2,h}
O^{g_1,h}_{l_1} O^{g_2,h}_{l_2} 
\\
&  \qquad
\sum_{i_1,j,k_1} X_{i_1j,k_1} \,
d^{k_1,g_1}_{l_1,i_1j}
\sum_{i_2,k_2} {\sqrt{\lambda_{k_2}} e_{i_2,k_2} \,\langle {\psi^{(\mathbf{R})}_{k_2}}(t) \,,\, \phi_{g_2}(t) \rangle_{H^1_{\mathbf{R}}}} \,
 {m}_{{l_2},i_2} \Big)
\\
& =
\sum_{l_1,l_2,g_1,g_2,h}
O^{g_1,h}_{l_1} O^{g_2,h}_{l_2} 
\\
&  \qquad	
\sum_{i_1,i_2,j} \sum_{k_1,k_2} \sqrt{\lambda_{k_2}} \;d^{k_1,g_1}_{l_1,i_1j}\; {m}_{{l_2},i_2}
\, E\big( X_{i_1j,k_1} e_{i_2,k_2} \big) \langle {\psi^{(\mathbf{R})}_{k_2}}(t) \,,\, \phi_{g_2}(t) \rangle_{H^1_{\mathbf{R}}}
\\
& =
\sum_{l_1,l_2,g_1,g_2,h}
O^{g_1,h}_{l_1} O^{g_2,h}_{l_2} 
\\
&  \qquad	
\sum_{i_1,i_2,j} \sum_{k_1,k_2} \delta_{i_1}^{i_2} \delta_{k_1}^{k_2} \lambda_{k_1}\;d^{k_1,g_1}_{l_1,i_1j}\; {m}_{{l_2},i_2}
\langle {\psi^{(\mathbf{R})}_{k_2}}(t) \,,\, \phi_{g_2}(t) \rangle_{H^1_{\mathbf{R}}}
\\
& =
\sum_{l_1,l_2,g_1,g_2,h}
O^{g_1,h}_{l_1} O^{g_2,h}_{l_2} 
\sum_{i,j} \sum_{k} \; d^{k,g_1}_{l_1,ij} \; {m}_{{l_2},i}
\lambda_{k} \langle {\psi^{(\mathbf{R})}_k}(t) \,,\, \phi_{g_2}(t) \rangle_{H^1_{\mathbf{R}}}
\\
& =
\sum_{l_1,l_2,g_1,g_2,h}
O^{g_1,h}_{l_1} O^{g_2,h}_{l_2} 
\sum_{i,j} \sum_{k} \; d^{k,g_1}_{l_1,ij} \; {m}_{{l_2},i}
\big(\lambda_{k} \langle \boldsymbol{\Psi}_{k}(t) \,,\, (\phi_{g_2}(t),\phi_{g_2}'(t))^T \rangle_{\mathbf{L}^2_{\mathbf{R}}} \big)
\\
& =
\sum_{l_1,l_2,g_1,g_2,h}
O^{g_1,h}_{l_1} O^{g_2,h}_{l_2} 
\sum_{i,j} \sum_{k} \; d^{k,g_1}_{l_1,ij} \; {m}_{{l_2},i}
\langle \boldsymbol{\Psi}_{k}(t) \,,\, (\phi_{g_2}(t),\phi_{g_2}'(t))^T \rangle_{\mathbf{L}^2}
\\
& = 0,
\end{align*}
the last equality being a consequence of \eqref{eq:r0matrix1}.

\bigskip

{\footnotesize{
\noindent\textbf{Acknowledgments}.
We thank an anonymous referee for his very useful comments which made us rethink more deeply this problem.}}


\begin{thebibliography}{24}
\providecommand{\natexlab}[1]{#1}
\providecommand{\url}[1]{\texttt{#1}}
\providecommand{\urlprefix}{URL }
\expandafter\ifx\csname urlstyle\endcsname\relax
  \providecommand{\doi}[1]{doi:\discretionary{}{}{}#1}\else
  \providecommand{\doi}[1]{doi:\discretionary{}{}{}\begingroup
  \urlstyle{rm}\url{#1}\endgroup}\fi
\providecommand{\bibinfo}[2]{#2}

\bibitem[{Ramsay and Silverman(2005)}]{Ramsay:Silverman05}
\bibinfo{author}{J.~O. Ramsay}, \bibinfo{author}{B.~W. Silverman},
  \bibinfo{title}{Functional data analysis}, Springer Series in Statistics,
  \bibinfo{publisher}{Springer, New York}, \bibinfo{edition}{second} edn., ISBN
  \bibinfo{isbn}{978-0387-40080-8; 0-387-40080-X}, \bibinfo{year}{2005}.

\bibitem[{Horv{\'a}th and Kokoszka(2012)}]{HorKok:HK12}
\bibinfo{author}{L.~Horv{\'a}th}, \bibinfo{author}{P.~Kokoszka},
  \bibinfo{title}{Inference for functional data with applications}, Springer
  Series in Statistics, \bibinfo{publisher}{Springer, New York}, ISBN
  \bibinfo{isbn}{978-1-4614-3654-6},
  \doi{\bibinfo{doi}{10.1007/978-1-4614-3655-3}}, \bibinfo{year}{2012}.

\bibitem[{Ferraty and Vieu(2006)}]{Ferr:V06}
\bibinfo{author}{F.~Ferraty}, \bibinfo{author}{P.~Vieu},
  \bibinfo{title}{Nonparametric functional data analysis}, Springer Series in
  Statistics, \bibinfo{publisher}{Springer, New York}, ISBN
  \bibinfo{isbn}{0-387-30369-3; 978-0387-30369-7}, \bibinfo{note}{theory and
  practice}, \bibinfo{year}{2006}.

\bibitem[{Sangalli et~al.(2009)Sangalli, Secchi, Vantini, and
  Veneziani}]{sangalli09}
\bibinfo{author}{L.~Sangalli}, \bibinfo{author}{P.~b. Secchi},
  \bibinfo{author}{S.~Vantini}, \bibinfo{author}{A.~Veneziani},
  \bibinfo{title}{Efficient estimation of three-dimensional curves and their
  derivatives by free-knot regression splines, applied to the analysis of inner
  carotid artery centrelines}, \bibinfo{journal}{Journal of the Royal
  Statistical Society. Series C: Applied Statistics}
  \bibinfo{volume}{58}~(\bibinfo{number}{3}) (\bibinfo{year}{2009})
  \bibinfo{pages}{285--306}, \bibinfo{note}{cited By 13}.

\bibitem[{Pigoli and Sangalli(2012)}]{pigoli12}
\bibinfo{author}{D.~Pigoli}, \bibinfo{author}{L.~Sangalli},
  \bibinfo{title}{Wavelets in functional data analysis: Estimation of
  multidimensional curves and their derivatives},
  \bibinfo{journal}{Computational Statistics and Data Analysis}
  \bibinfo{volume}{56}~(\bibinfo{number}{6}) (\bibinfo{year}{2012})
  \bibinfo{pages}{1482--1498}, \bibinfo{note}{cited By 4}.

\bibitem[{Baraldo et~al.(2013)Baraldo, Ieva, Paganoni, and Vitelli}]{Baraldo13}
\bibinfo{author}{S.~Baraldo}, \bibinfo{author}{F.~Ieva}, \bibinfo{author}{A.~M.
  Paganoni}, \bibinfo{author}{V.~Vitelli}, \bibinfo{title}{Outcome prediction
  for heart failure telemonitoring via generalized linear models with
  functional covariates}, \bibinfo{journal}{Scand. J. Stat.}
  \bibinfo{volume}{40}~(\bibinfo{number}{3}) (\bibinfo{year}{2013})
  \bibinfo{pages}{403--416}, ISSN \bibinfo{issn}{0303-6898},
  \doi{\bibinfo{doi}{10.1111/j.1467-9469.2012.00818.x}}.

\bibitem[{Hall et~al.(2009)Hall, M{\"u}ller, and Yao}]{Hall09}
\bibinfo{author}{P.~Hall}, \bibinfo{author}{H.-G. M{\"u}ller},
  \bibinfo{author}{F.~Yao}, \bibinfo{title}{Estimation of functional
  derivatives}, \bibinfo{journal}{Ann. Statist.}
  \bibinfo{volume}{37}~(\bibinfo{number}{6A}) (\bibinfo{year}{2009})
  \bibinfo{pages}{3307--3329}, ISSN \bibinfo{issn}{0090-5364},
  \doi{\bibinfo{doi}{10.1214/09-AOS686}}.

\bibitem[{Aletti et~al.(2014)Aletti, May, and Tommasi}]{IWFOS}
\bibinfo{author}{G.~Aletti}, \bibinfo{author}{C.~May},
  \bibinfo{author}{C.~Tommasi}, \bibinfo{title}{Optimal designs for linear
  models with functional responses}, in: \bibinfo{editor}{E.~G. Bongiorno},
  \bibinfo{editor}{E.~Salinelli}, \bibinfo{editor}{A.~Goia},
  \bibinfo{editor}{P.~Vieu} (Eds.), \bibinfo{booktitle}{Contributions in
  Infinite-Dimensional Statistics and Related Topics},
  \bibinfo{publisher}{Societ{\`a} Editrice Esculapio}, ISBN
  \bibinfo{isbn}{978-8874-887637}, \bibinfo{pages}{19--24},
  \doi{\bibinfo{doi}{10.15651/978-88-748-8763-7}}, \bibinfo{year}{2014}.

\bibitem[{Menafoglio et~al.(2013)}]{Menafoglio13}
\bibinfo{author}{A.~Menafoglio}, \bibinfo{author}{P.~Secchi}, \bibinfo{author}{M.~Dalla~Rosa}, \bibinfo{title}{A
  universal kriging predictor for spatially dependent functional data of a
  {H}ilbert space}, \bibinfo{journal}{Electron. J. Stat.} \bibinfo{volume}{7}
  (\bibinfo{year}{2013}) \bibinfo{pages}{2209--2240}, ISSN
  \bibinfo{issn}{1935-7524}, \doi{\bibinfo{doi}{10.1214/13-EJS843}}.

\bibitem[{Kufner(1985)}]{Kufner85}
\bibinfo{author}{A.~Kufner}, \bibinfo{title}{Weighted {S}obolev spaces}, A
  Wiley-Interscience Publication, \bibinfo{publisher}{John Wiley \& Sons, Inc.,
  New York}, ISBN \bibinfo{isbn}{0-471-90367-1}, \bibinfo{note}{translated from
  the Czech}, \bibinfo{year}{1985}.

\bibitem[{Kiefer(1974)}]{kiefer74}
\bibinfo{author}{J.~Kiefer}, \bibinfo{title}{General equivalence theory for
  optimum designs (approximate theory)}, \bibinfo{journal}{Ann. Statist.}
  \bibinfo{volume}{2} (\bibinfo{year}{1974}) \bibinfo{pages}{849--879}, ISSN
  \bibinfo{issn}{0090-5364}.

\bibitem[{Fedorov(1972)}]{Fedorov72}
\bibinfo{author}{V.~V. Fedorov}, \bibinfo{title}{Theory of optimal
  experiments}, \bibinfo{publisher}{Academic Press, New York-London},
  \bibinfo{note}{translated from the Russian and edited by W. J. Studden and E.
  M. Klimko, Probability and Mathematical Statistics, No. 12},
  \bibinfo{year}{1972}.

\bibitem[{Pukelsheim(1993)}]{Pukel93}
\bibinfo{author}{F.~Pukelsheim}, \bibinfo{title}{Optimal design of
  experiments}, Wiley Series in Probability and Mathematical Statistics:
  Probability and Mathematical Statistics, \bibinfo{publisher}{John Wiley \&
  Sons, Inc., New York}, ISBN \bibinfo{isbn}{0-471-61971-X}, \bibinfo{note}{a
  Wiley-Interscience Publication}, \bibinfo{year}{1993}.

\bibitem[{Silvey(1980)}]{Silvey}
\bibinfo{author}{S.~D. Silvey}, \bibinfo{title}{Optimal design},
  \bibinfo{publisher}{Chapman \& Hall, London-New York}, ISBN
  \bibinfo{isbn}{0-412-22910-2}, \bibinfo{note}{an introduction to the theory
  for parameter estimation, Monographs on Applied Probability and Statistics},
  \bibinfo{year}{1980}.

\bibitem[{Shen and Faraway(2004)}]{Shen04}
\bibinfo{author}{Q.~Shen}, \bibinfo{author}{J.~Faraway}, \bibinfo{title}{An
  {$F$} test for linear models with functional responses},
  \bibinfo{journal}{Statist. Sinica} \bibinfo{volume}{14}~(\bibinfo{number}{4})
  (\bibinfo{year}{2004}) \bibinfo{pages}{1239--1257}, ISSN
  \bibinfo{issn}{1017-0405}.

\bibitem[{Atkinson et~al.(2007)Atkinson, Donev, and Tobias}]{Atkinson07}
\bibinfo{author}{A.~C. Atkinson}, \bibinfo{author}{A.~N. Donev},
  \bibinfo{author}{R.~D. Tobias}, \bibinfo{title}{Optimum experimental designs,
  with {SAS}}, vol.~\bibinfo{volume}{34} of \emph{\bibinfo{series}{Oxford
  Statistical Science Series}}, \bibinfo{publisher}{Oxford University Press,
  Oxford}, ISBN \bibinfo{isbn}{978-0-19-929660-6}, \bibinfo{year}{2007}.

\bibitem[{Bosq(2000)}]{Bosq00}
\bibinfo{author}{D.~Bosq}, \bibinfo{title}{Linear processes in function
  spaces}, vol. \bibinfo{volume}{149} of \emph{\bibinfo{series}{Lecture Notes
  in Statistics}}, \bibinfo{publisher}{Springer-Verlag, New York}, ISBN
  \bibinfo{isbn}{0-387-95052-4},
  \doi{\bibinfo{doi}{10.1007/978-1-4612-1154-9}}, \bibinfo{note}{theory and
  applications}, \bibinfo{year}{2000}.

\bibitem[{Marley and Woods(2010)}]{Marley10}
\bibinfo{author}{C.~J. Marley}, \bibinfo{author}{D.~C. Woods},
  \bibinfo{title}{A comparison of design and model selection methods for
  supersaturated experiments}, \bibinfo{journal}{Comput. Statist. Data Anal.}
  \bibinfo{volume}{54}~(\bibinfo{number}{12}) (\bibinfo{year}{2010})
  \bibinfo{pages}{3158--3167}, ISSN \bibinfo{issn}{0167-9473},
  \doi{\bibinfo{doi}{10.1016/j.csda.2010.02.017}}.

\bibitem[{Marley(2011)}]{Marley2011}
\bibinfo{author}{C.~J. Marley}, \bibinfo{title}{Screening experiments using
  supersaturated designs with application to industry}, Ph.D. thesis,
  \bibinfo{school}{University of Southampton},
  \urlprefix\url{http://eprints.soton.ac.uk/176451}, \bibinfo{year}{2011}.

\bibitem[{Woods et~al.(2013)Woods, Marley, and Lewis}]{Woods13}
\bibinfo{author}{D.~C. Woods}, \bibinfo{author}{C.~J. Marley},
  \bibinfo{author}{S.~M. Lewis}, \bibinfo{title}{Designed experiments for
  semi-parametric models and functional data with a case-study in Tribology},
  in: \bibinfo{booktitle}{World Statistics Congress, Hong Kong},
  \bibinfo{year}{2013}.

\bibitem[{Chiou et~al.(2004)Chiou, M{\"u}ller, and Wang}]{Chiou04}
\bibinfo{author}{J.-M. Chiou}, \bibinfo{author}{H.-G. M{\"u}ller},
  \bibinfo{author}{J.-L. Wang}, \bibinfo{title}{Functional response models},
  \bibinfo{journal}{Statist. Sinica} \bibinfo{volume}{14}~(\bibinfo{number}{3})
  (\bibinfo{year}{2004}) \bibinfo{pages}{675--693}, ISSN
  \bibinfo{issn}{1017-0405}.

\bibitem[{Shen and Xu(2007)}]{ShenXu}
\bibinfo{author}{Q.~Shen}, \bibinfo{author}{H.~Xu}, \bibinfo{title}{Diagnostics
  for linear models with functional responses},
  \bibinfo{journal}{Technometrics} \bibinfo{volume}{49}~(\bibinfo{number}{1})
  (\bibinfo{year}{2007}) \bibinfo{pages}{26--33}, ISSN
  \bibinfo{issn}{0040-1706}, \doi{\bibinfo{doi}{10.1198/004017006000000444}}.

\bibitem[{Fedorov and Hackl(1997)}]{FedorovHackl}
\bibinfo{author}{V.~V. Fedorov}, \bibinfo{author}{P.~Hackl},
  \bibinfo{title}{Model-oriented design of experiments}, vol.
  \bibinfo{volume}{125} of \emph{\bibinfo{series}{Lecture Notes in
  Statistics}}, \bibinfo{publisher}{Springer-Verlag, New York}, ISBN
  \bibinfo{isbn}{0-387-98215-9},
  \doi{\bibinfo{doi}{10.1007/978-1-4612-0703-0}}, \bibinfo{year}{1997}.

\bibitem[{Perrin et~al.(2013)Perrin, Soize, Duhamel, and
  Funfschilling}]{Perrin13}
\bibinfo{author}{G.~Perrin}, \bibinfo{author}{C.~Soize},
  \bibinfo{author}{D.~Duhamel}, \bibinfo{author}{C.~Funfschilling},
  \bibinfo{title}{Karhunen-{L}o\`eve expansion revisited for vector-valued
  random fields: scaling, errors and optimal basis}, \bibinfo{journal}{J.
  Comput. Phys.} \bibinfo{volume}{242} (\bibinfo{year}{2013})
  \bibinfo{pages}{607--622}, ISSN \bibinfo{issn}{0021-9991},
  \doi{\bibinfo{doi}{10.1016/j.jcp.2013.02.036}}.

\end{thebibliography}
\end{document}